\documentclass[a4paper,12pt]{article}

\usepackage[utf8]{inputenc}
\usepackage[francais,english]{babel}
\usepackage{lmodern}
\usepackage[T1]{fontenc}

\usepackage{hyperref}
\usepackage{graphicx}
\usepackage{amsmath}
\usepackage{amsfonts}
\usepackage{amssymb}
\usepackage{amsthm}
\usepackage{xcolor}
\usepackage{pdfsync}
\usepackage{datetime}

\usepackage[text={7in,10in},centering]{geometry}
\usepackage[margin=60pt,format=plain,font=footnotesize,labelsep=period]{caption}

\newcommand{\bpm}{\begin{pmatrix}}
\newcommand{\epm}{\end{pmatrix}}

\theoremstyle{definition}
\newtheorem{conjecture}{{\bf Conjecture}}
\newtheorem{theorem}{{\bf Theorem}}

\newtheorem{remark}{{\bf Remark}}
\newtheorem{lemma}{{\bf Lemma}}

\newcommand{\real}{\mathbb{R}}

\newcommand{\N}{\mathbb{N}}
\newcommand{\mb}{\mathbf}

\newcommand{\mbK}{\mathbf{K}}

\newcommand{\mc}{\mathcal}

\newcommand{\tmy}{\textrm{y}}

\newcommand{\rp}{\text{Re}}
\newcommand{\ip}{\text{Im}}

\newcommand{\ls}{\leqslant}
\newcommand{\gs}{\geqslant}

\newcommand{\hh}{\hspace{1cm}}
\newcommand{\hd}{\hspace{5mm}}

\newcommand{\red}{\textcolor{red}}

\definecolor{fgreen}{RGB}{29, 177, 0}
\definecolor{darkgreen}{rgb}{0,0.7,0}

\newcommand{\noi}{\noindent}

\newcommand{\p}{\partial}

\title{Diffusion and robustness of boundary feedback stabilization of hyperbolic systems \\[1em] \Large{ \textit{\`A notre maitre et ami Eduardo}}}

\author{Georges Bastin\thanks{Department of Mathematical Engineering, ICTEAM, UCLouvain, Louvain-La-Neuve, Belgium. georges.bastin@uclouvain.be} \,,\, Jean-Michel Coron\thanks{Laboratoire Jacques-Louis Lions, Sorbonne Universit\'{e}, Universit\'{e} de Paris, CNRS, INRIA, \'{e}quipe Cage, Paris, France. jean-michel.coron@sorbonne-universite.fr}\,\, and Amaury Hayat\thanks{CERMICS, \'{E}cole des Ponts ParisTech, 6 - 8, Avenue Blaise Pascal, Cit\'{e} Descartes - Champs sur Marne, 77455 Marne la Vall\'{e}e, France. amaury.hayat@enpc.fr.}}

\date{\empty}

\begin{document}

\maketitle

\abstract{We consider the problem of boundary feedback control of single-input-single-output (SISO) one-dimensional linear hyperbolic systems when sensing and actuation are anti-located. The main issue of the output feedback stabilization is that it requires dynamic control laws that include delayed values of the output (directly or through state observers) which may not be robust to infinitesimal uncertainties on the characteristic velocities. The purpose of this paper is to highlight some features of this problem by addressing the feedback stabilization of an unstable open-loop system which is made up of two interconnected transport equations and provided with anti-located boundary sensing and actuation. The main contribution is to show that the robustness of the control against delay uncertainties is recovered as soon as an arbitrary small diffusion is present in the system. Our analysis also reveals that the effect of diffusion on stability is far from being an obvious issue by exhibiting an alternative simple example where the presence of diffusion has a destabilizing effect instead.}

\section{Introduction}

The output feedback stabilization of single-input-single-output (SISO) one-dimensional linear hyperbolic systems is a subject that has been widely studied in the scientific literature since the nineties when both actuation and sensing are located at the boundaries. In the case where the control input and the measured output are co-located at the same boundary, the problem is now relatively well understood and has given rise to numerous publications,  both in the linear case \cite{Aam13, AnfAam17b} and in the nonlinear case \cite{BasCor14, Gug15}, in particular in fluid mechanics for Saint-Venant equations \cite{SanPri08, BasCor14, HaySha19} or Euler equations \cite{GugLeuTam12}, to name just a few of the many publications on the subject.

In contrast, when the actuator acts through one boundary, whereas the sensor is placed at the other boundary, the output feedback stabilization problem can become much more complicated and remains largely unexplored in the literature. As Krstic et al. pointed out in \cite{KrsGuoBal08}, the difficulty arises from the fact that ``the input-output operator is no longer passive (...) which precludes the application of simple controllers''. Anti-located sensing and actuation requires to use dynamic compensators that include delayed values of the output (directly or through state observers). A meaningful example is the output feedback stabilization of a simple unstable wave equation addressed in \cite{KrsGuoBal08} using a separation principle that combines a state feedback control with a state observer. This approach is extended to the adaptive stabilization of more general linear hyperbolic systems with unknown parameters in \cite{AnfAam17} and \cite{BerKrs14b}. Recently, an experimental application to the control of hydraulic waves is reported in \cite{WurMayWoi21}, where the actuation is provided by a moving boundary while the water level is measured at the other boundary.

We can also mention references \cite{AurDiMeg, HuMegVaz15} which deal with MIMO systems with anti-located multivariable sensors and actuators for $n \times n$ linear hyperbolic systems expressed in a characteristic form having a very specific input/output structure.

It is important to note that the use of control laws containing delayed output feedback or state observers can however prove to be problematic because it can be strongly sensitive to uncertainties on the characteristic velocities of the plant model \cite{MicNic07, Fri14, ChiMazSig16}. This happens when the control design relies on a (supposed) exact knowledge of some characteristic velocities that must be exactly compensated in the control law such that the stability can be destroyed by arbitrarily small modelling uncertainties.

In this paper, our purpose is to highlight some features of this problem by addressing the feedback stabilization of an unstable open-loop system which is made up of two interconnected transport equations and provided with anti-located boundary sensing and actuation.

Our paper is organized as follows. The control problem is described in Section \ref{sectioncontrolproblem}. It is first shown that the considered  control system is open loop unstable and cannot be stabilized by a simple proportional output feedback.  Then, it is shown that the system can be stabilized by a dynamic controller that involves a delayed output feedback. However, this control turns out not to be robust with respect to delay uncertainties precisely because the control requires a (utopian) exact knowledge of the transport velocity.

The main contribution of this paper is to show that the robustness of the control against delay uncertainties is recovered as soon as an arbitrary small diffusion is present in the system. For that purpose, in Section \ref{sectionopenloop}, it is first assumed that the considered plant is subject to a slight phenomenon of diffusion, interpreted as a viscosity and the corresponding (unstable) input-output transfer function is computed. Then in Sections \ref{sectionclosedloop} and \ref{sectionrobustness}, we show that the dynamic output (non robust) feedback designed for the inviscid case also stabilizes exponentially the viscous system when the (unknown) diffusion is small, and that, in this case, the control proves to be perfectly robust, even if the diffusion is almost negligible. Interestingly, an upper bound on the decay rate appears when adding a small viscosity, and this upper bound is uniform with respect to the diffusion parameter $\eta$ when it is small, while for the unperturbed system with $\eta=0$ the decay rate is infinite (the system is finite-time stable).

Our analysis in Sections \ref{sectionclosedloop} and \ref{sectionrobustness} also reveals that the effect of diffusion on stability is far from being an obvious issue, contrary to what one might expect. It is indeed well known that, in hyperbolic systems, the presence of diffusion (or friction) can have a destabilizing as well as a stabilizing effect (see for instance the references \cite{FraCol15, Tur52}). This issue is further discussed in Section \ref{sectionscalar} where we present an example of another simple hyperbolic system which simplifies the previous case and for which, however, the same diffusion term destroys the stability instead of strengthening it.

Some final conclusions are given in Section \ref{sectionconclusion}.

\section{Description of the control problem} \label{sectioncontrolproblem}

We consider the open loop control system represented in Figure \ref{openlooptransport}. The system is made up of the positive feedback interconnection of two identical transport systems. The system dynamics are described in the time domain by the following equations:
\begin{subequations} \label{invisys}
\begin{gather}
\p_t y_1(t,x) + \upsilon \p_x y_1(t,x) = 0, \label{invisys1}\\[0.5em]
\p_t y_2(t,x) + \upsilon \p_x y_2(t,x) = 0, \label{invisys2}\\[0.5em]
y_1(t,0) = y_2(t,1) + U(t), \label{bcinvisys1}\\[0.5em]
y_2(t,0) = y_1(t,1), \label{bcsys2}\\[0.5em]
Y(t) = y_1(t,1). \label{bcinvisys4}
\end{gather}
\end{subequations}
where $U(t)$ is the control input and $Y(t)$ is the measurable output. In the classical pure transport equations \eqref{invisys1} and \eqref{invisys2}, the parameter $\upsilon > 0$ denotes the transport velocity.
\begin{figure}[ht]
  \includegraphics[width=0.60\textwidth]{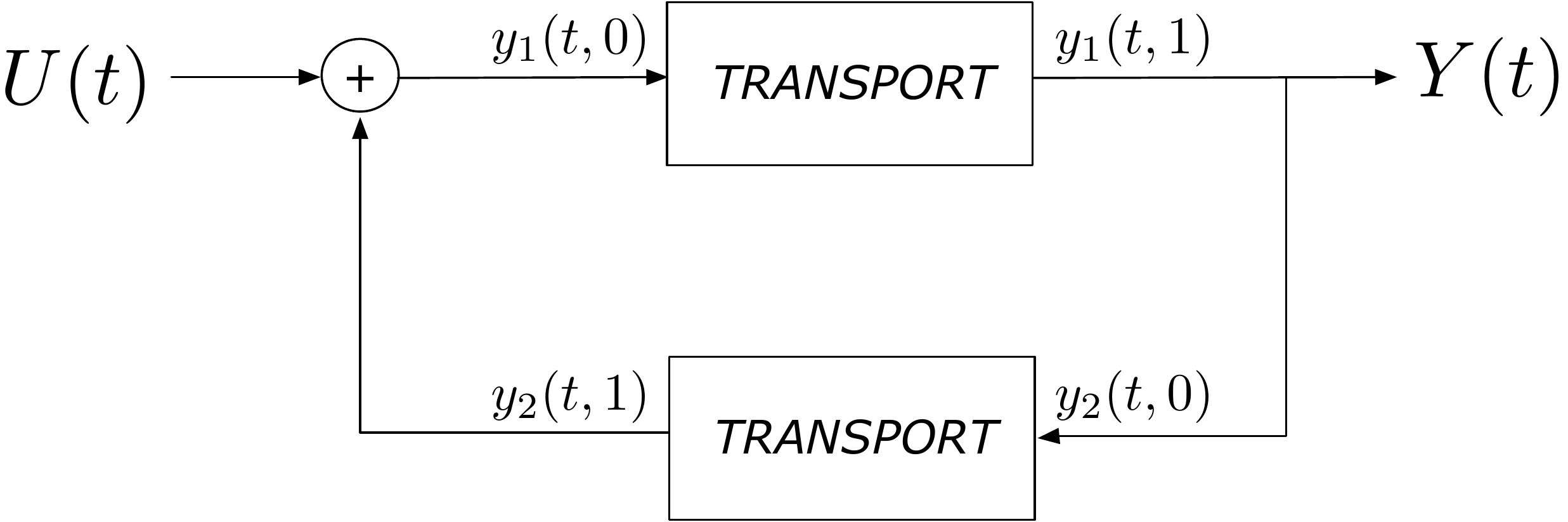}
  \centering
  \caption{The structure of the open loop control system considered in this paper.}
  \label{openlooptransport}
\end{figure}

In the frequency domain, it is well known that the transfer function of the transport systems \eqref{invisys1} and \eqref{invisys2} is
\begin{equation} \label{tfdelay}
f_o(s) = e^{-s\tau} \hd \text{with the time delay} \hd \tau = \dfrac{1}{\upsilon},	
\end{equation}
where $s$ denotes the Laplace complex variable.

It follows that the overall input-output transfer function of the open loop system \eqref{invisys} is
\begin{equation}
G(s) = \dfrac{\textrm{Y}(s)}{\textrm{U}(s)} = \dfrac{f_o(s)}{1 - f_o^2(s)} = \dfrac{e^{s\tau}}{e^{2s\tau} - 1},	\label{transferinviscid} 	
\end{equation}
where $\textrm{Y}(s)$ and $\textrm{U}(s)$ denote the Laplace transforms of the output $Y(t)$ and the input $U(t)$ respectively.

The poles of the system are the roots of the characteristic equation
\begin{equation}
 e^{2s\tau} - 1 = 0.	
\end{equation}
This open loop control system is clearly not asymptotically stable since all the poles are  located on the imaginary axis.

In order to illustrate the challenge that arises when actuation and sensing are anti-located, we shall first show that, despite its apparent simplicity, this unstable system cannot be stabilized with a simple proportional output feedback, i.e. with a static proportional controller of the form
\begin{equation}
U(t) = - 2k_p Y(t)	\label{propcontrol}
\end{equation}
where $k_p \neq 0$ is a control tuning parameter.

In the frequency domain, for the system \eqref{transferinviscid} with the control law \eqref{propcontrol} the characteristic equation of the closed loop system is:
\begin{equation}
e^{2s\tau} +  2k_pe^{s\tau} - 1 = 0.
\end{equation}
Solving this equation for $e^{s\tau}$, we get
\begin{equation}
e^{s\tau} = -k_p \pm \sqrt{1 + k_p^2}.
\end{equation}
Then for any $k_p \neq 0$ there is an infinity of system poles $\sigma + i\omega$ lying on two vertical lines with real parts:
\begin{equation}
\sigma =  \upsilon \ln \left( \sqrt{1 + k_p^2} + |k_p| \right) > 0 \hd \text{and} \hd \sigma = \upsilon \ln \left( \sqrt{1 + k_p^2} - |k_p| \right) < 0.
\end{equation}
It follows that the unstable system \eqref{transferinviscid} cannot be stabilized with the static controller \eqref{propcontrol}.

Let us now show that the  system can actually be stabilized by a dynamic controller that involves a delayed output feedback.

From \eqref{transferinviscid}, it follows that the input-output dynamics of the system \eqref{invisys} in the time domain can alternatively be represented by the delay-difference equation
\begin{equation}
Y(t) - Y(t-2\tau) =  U(t-\tau). \label{inoutinviscid}
\end{equation}
A simple and natural candidate for a stabilizing feedback control law is then
\begin{equation} \label{dyncontrol}
U(t) = -2k_1 Y(t) - k_2Y(t-\tau)	
\end{equation}
where $k_1$ and $k_2$ are control tuning parameters.
With this control law, the closed loop dynamics are
\begin{equation} \label{closinviscid}
	Y(t) + 2k_1 Y(t-\tau) +  (k_2 - 1)Y(t-2\tau) = 0.
\end{equation}
Clearly, we can conclude that the controller \eqref{dyncontrol} exponentially stabilizes the  system \eqref{invisys} if the tuning parameters $k_1$ and $k_2$ are selected such that the roots of the polynomial
\begin{equation} \label{poly}
w^2 + 2k_1 w + (k_2 - 1)
\end{equation}
 are located inside the unit circle. However, this should be considered with caution because it is well known that boundary feedback stabilization of hyperbolic systems with delayed control may be sensitive to delay modeling errors (see e.g. \cite{MicNic07}). To clarify this point we consider the time domain representation of the dynamical control law \eqref{dyncontrol} defined as
\begin{gather}
 \label{dyncontrol2}
\begin{split}
&\p_t \hat y_2(t,x) + \upsilon \p_x \hat y_2(t,x) = 0,\\
&\hat y_{2}(t,0)= y_{1}(t,1),
\end{split}\\[0.5em]
U(t) = -2 k_1 y_1(t,1) - k_2 \hat y_2(t,1), \label{dyncontrol3}
\end{gather}
where the transport equation \eqref{dyncontrol2} with transport velocity $\upsilon = 1/\tau$ is equivalent to  the time delay $\tau$ of the control law \eqref{dyncontrol}. With this definition, the closed loop system \eqref{invisys}, \eqref{dyncontrol2}, \eqref{dyncontrol3} is then represented as follows:
\begin{subequations} \label{closinvisys}
\begin{gather}
\p_t y_1(t,x) + \upsilon \p_x y_1(t,x)= 0, \label{closinvisys1}\\[0.5em]
\p_t y_2(t,x) + \upsilon \p_x y_2(t,x) = 0, \label{closinvisys2}\\[0.5em]
\p_t \hat y_2(t,x) + \upsilon \p_x \hat y_2(t,x) = 0, \label{closinvisys3}\\[0.5em]
\bpm y_1(t,0) \\[0.5em] y_2(t,0) \\[0.5em] \hat y_2(t,0) \epm = \underbrace{\bpm -2k_1 & 1 & - k_2 \\[0.5em] 1 & 0 & 0 \\[0.5em]\ 1 & 0 & 0\epm }_{\displaystyle \mbK} \bpm y_1(t,1) \\[0.5em] y_2(t,1) \\[0.5em] \hat y_2(t,1) \epm. \label{closinvisysbc}
\end{gather}
\end{subequations}

Now, as proved in Appendix \ref{appendixrho}, for the matrix $\mbK$ defined in \eqref{closinvisysbc}, it can be shown that
\begin{equation}
	 \bar \rho(\mbK) = |k_1| + \sqrt{1 + k_1^2 + |k_2|} \; \gs 1 \hd \text{for all } (k_1, k_2)\in\mathbb{R}^{2},
\end{equation}
where $\bar \rho(\mbK)$ is defined  as follows:
\begin{equation}
\bar \rho(\mbK) :=\max \{ \rho(\text{diag}\!\left\{e^{-i\theta_1}, e^{-i\theta_2},  e^{-i\theta_3}\right\}\mbK) ; (\theta_1, \theta_2 , \theta_3)^{\! \mathsf{T}} \in \real^3 \},
\end{equation}
 $\rho(M)$ denoting the spectral radius of the matrix $M$.
By \cite{Sil76} (see also \cite[Chapter 9, Theorem 6.1]{HalVer93} and \cite[Chapter 3]{BasCor14}), we know that $\bar \rho(\mbK) < 1$ is a necessary (and sufficient) condition to have a stability which is robust against small uncertainties in the characteristic velocities. Although the ideal closed loop system \eqref{closinvisys} is exponentially stable (with all the poles strictly located in the left half complex plane provided $k_{1}$ and $k_{2}$ are chosen accordingly), the stability can be destroyed by an arbitrarily small difference in characteristic velocities between the plant equations \eqref{closinvisys1}, \eqref{closinvisys2} and the controller equation \eqref{closinvisys3}. More precisely, if we assume that the physical transport velocities are $\upsilon + \varepsilon_1$, $\upsilon + \varepsilon_2$ with $\varepsilon_1, \varepsilon_2$ representing uncertainties in the plant equations \eqref{closinvisys1}, \eqref{closinvisys2} rewritten as follows:
\begin{subequations} \label{closinvisysperturb}
\begin{gather}
\p_t y_1(t,x) + (\upsilon + \varepsilon_1) \p_x y_1(t,x)= 0, \label{closinvisys1perturb}\\[0.5em]
\p_t y_2(t,x) + (\upsilon + \varepsilon_2) \p_x y_2(t,x) = 0, \label{closinvisys2perturb}
\end{gather}
\end{subequations}
then the closed loop system \eqref{closinvisys} with \eqref{closinvisys1}, \eqref{closinvisys2} replaced by \eqref{closinvisys1perturb}, \eqref{closinvisys2perturb} may become unstable, with poles moving to the right half complex plane even for arbitrarily small $\varepsilon_i$  perturbations.

Should this necessarily mean that the control law \eqref{dyncontrol} with delayed feedback could not be applied in practice? Our objective, in this paper, is exactly to prove the opposite! Indeed, our main contribution will be to show that, even with the simple control law \eqref{dyncontrol}, the robustness of the output feedback stabilization against delay uncertainties can be recovered as soon as an arbitrary small diffusion is present in the system. Our analysis will also reveal, however, that the effect on stability of adding an arbitrarily small diffusion is far from obvious, contrary to what one might expect. Indeed, while diffusion strengthens the robustness of the exponential stability for the $2\times 2$ problem \eqref{closinvisysperturb}, it can also destroy the stability of similar simpler systems as we will see in Section \ref{sectionscalar}.

We shall consider the special case of a dead beat control where $k_1 = 0$ and $k_2 = 1$. In that case the characteristic equation of the closed loop inviscid system reduces to $e^{2s\tau} = 0$, meaning that all the poles of the system have negative real parts that are moved off to infinity. However, for that system we have
\begin{equation}
\bar \rho(\mbK) = \sqrt{2}	
\end{equation}
showing a strict lack of robustness of the control w.r.t. delay inaccuracy.

\section{The open loop control system with diffusion}\label{sectionopenloop}

We consider again a control system as represented in Figure \ref{openlooptransport}. However, we assume here that the two transport systems are subject to a slight phenomenon of diffusion.  The system is therefore made up of the feedback interconnection of two identical transport systems that are perturbed by a diffusion term which can be interpreted for example as a viscosity in the case of a fluid. For simplicity and without loss of generality, we assume a unit nominal transport velocity $\upsilon = 1$. The dynamics of the open loop control system are therefore described in the time domain by the following equations:
\begin{subequations} \label{unsys}
\begin{gather}
\p_t y_1(t,x) + \p_x y_1(t,x) - \eta \p^2_{xx} y_1(t,x) = 0, \label{unsys1}\\[0.5em]
\p_t y_2(t,x) + \p_x y_2(t,x) - \eta \p^2_{xx} y_2(t,x) = 0, \label{unsys2}\\[0.5em]
y_1(t,0) = y_2(t,1) + U(t), \label{bcunsys1}\\[0.5em]
y_2(t,0) = y_1(t,1), \label{bcunsys2}\\[0.5em]
\p_x y_1(t,1) = \p_x y_2(t,1) = 0 \label{bcunsys3}\\[0.5em]
Y(t) = y_1(t,1). \label{bcunsys4}
\end{gather}
\end{subequations}
As above $U(t)$ is the control input and $Y(t)$ is the measurable output while $\eta > 0$ is the viscosity coefficient.

In the frequency domain, with $\tmy_{\!1}(s,x)$ and $\tmy_{2}(s,x)$ denoting the Laplace transforms of $y_{1}(t,x)$ and $y_{2}(t,x)$,  the system is written
\begin{subequations} \label{fsys}
\begin{gather}
s\tmy_{\!1}(s,x) + \p_x \tmy_{\!1}(s,x) - \eta \p^2_{xx} \tmy_{\!1}(s,x) = 0, \label{fsys1}\\[0.5em]
s\tmy_2(s,x) + \p_x \tmy_2(s,x) - \eta \p^2_{xx} \tmy_2(s,x) = 0, \label{fsys2}\\[0.5em]
\tmy_{\!1}(s,0) = \tmy_2(s,1) + \textrm{U}(s), \label{bcfsys1}\\[0.5em]
\tmy_2(s,0) = \tmy_{\!1}(s,1), \label{bcfsys2}\\[0.5em]
\p_x \tmy_{\!1}(s,1) = \p_x \tmy_2(s,1) = 0 \label{bcfsys3}\\[0.5em]
\textrm{Y}(s) = \tmy_{\!1}(s,1). \label{bcfsys4}
\end{gather}
\end{subequations}
For any value of s, the solutions of the differential equations \eqref{fsys1}, \eqref{fsys2} are written
\begin{equation}
\tmy_i(s,x)= A_i(s) e^{\lambda_1(s)x} + B_i(s) e^{\lambda_2(s)x}, \hd i= 1,2,\label{solution}
\end{equation}
where $\lambda_1(s)$, $\lambda_2(s)$ are the roots of the polynomial
\begin{equation}
\eta \lambda^2 - \lambda - s = 0,	
\end{equation}
which implies that
\begin{equation} \label{lambda12}
\lambda_1(s) = \dfrac{1 +\sqrt{1 + 4 \eta s}}{2\eta}, \hd \lambda_2(s) = \dfrac{1 -\sqrt{1 + 4 \eta s}}{2\eta}.
\end{equation}
In \eqref{lambda12} and in the following $\sqrt{\cdot}$ denotes the principal value of the square root, which is well defined except when $1 + 4 \eta s\in \mathbb{R}_{-}$; however this particular case implies that $s$ is real and $s\leq -1/4\eta$ and therefore is negative and converges to $-\infty$ when $\eta\rightarrow 0^{+}$. We will see later on that this case can be considered separately.
Using the solution \eqref{solution} and the boundary condition \eqref{bcfsys3}, we have for each $i=1,2$,
\begin{gather}
A_i(s) + B_i(s) = \tmy_i(s,0), \\[0.5em]
\lambda_1(s)A_i(s)e^{\lambda_1(s)} + \lambda_2(s)B_i(s)e^{\lambda_2(s)}	= 0, \\[0.5em]
\tmy_i(s,1) = A_{i}(s)e^{\lambda_1(s)} + B_{i}(s)e^{\lambda_2(s)}.
\end{gather}
Eliminating $A_i(s)$ and $B_i(s)$ between these three equations, we get the transfer function $f_\eta(s)$ of each viscous transport system:
\begin{equation} \label{transfer}
f_\eta(s) = \dfrac{\tmy_i(s,1)}{\tmy_i(s,0)} = \dfrac{\lambda_1(s) - \lambda_2(s)}{\lambda_1(s) e^{-\lambda_2(s)} - \lambda_2(s) e^{-\lambda_1(s)}}.
\end{equation}
Remark that in the notation $f_\eta$, we use a subscript to emphasize the dependency on the viscosity parameter $\eta$. Remark also that in the limit, in the absence of a viscosity term (i.e. $\eta=0$), we recover the transfer function \eqref{tfdelay}: $f_o(s) = e^{-s\tau}$.

It follows that the system \eqref{fsys} is equivalent to:
\begin{subequations} \label{fsystf}
\begin{gather}
\tmy_{\!1}(s,1) = f_\eta(s) \tmy_{\!1}(s,0), \label{fsystf1}\\[0.5em]
\tmy_2(s,1) = f_\eta(s) \tmy_2(s,0), \label{fsystf2}\\[0.5em]
\tmy_{\!1}(s,0) = \tmy_2(s,1) + \textrm{U}(s), \label{bcfsystf1}\\[0.5em]
\tmy_2(s,0) = \tmy_{\!1}(s,1), \label{bcfsystf2}\\[0.5em]
\textrm{Y}(s) = \tmy_{\!1}(s,1). \label{bcfsystf4}
\end{gather}
\end{subequations}
Eliminating $\tmy_i(s,0), \tmy_i(s,1), (i=1,2),$ between these equations, we get that the overall input-output transfer function of the open loop system \eqref{unsys} is
\begin{equation}
G(s) = \dfrac{\textrm{Y}(s)}{\textrm{U}(s)} = \dfrac{f_\eta(s)}{1 - f_\eta^2(s)}.	
\end{equation}
The poles of the system are the roots of the characteristic equation
\begin{equation}
f_\eta^2(s) - 1 = 0.	
\end{equation}
In particular, it can be checked that $f_\eta(0) = 1$  for all $\eta \neq 0$, which means that there is a pole at the origin. Therefore, as in the inviscid case, the open loop system \eqref{unsys} is not asymptotically stable whatever the value of the viscosity $\eta$. As a matter of illustration, in Figure \ref{spectre-open1}, we present the spectrum of the system for $\eta = 0.1$.
\begin{figure}[t]
  \includegraphics[width=0.60\textwidth]{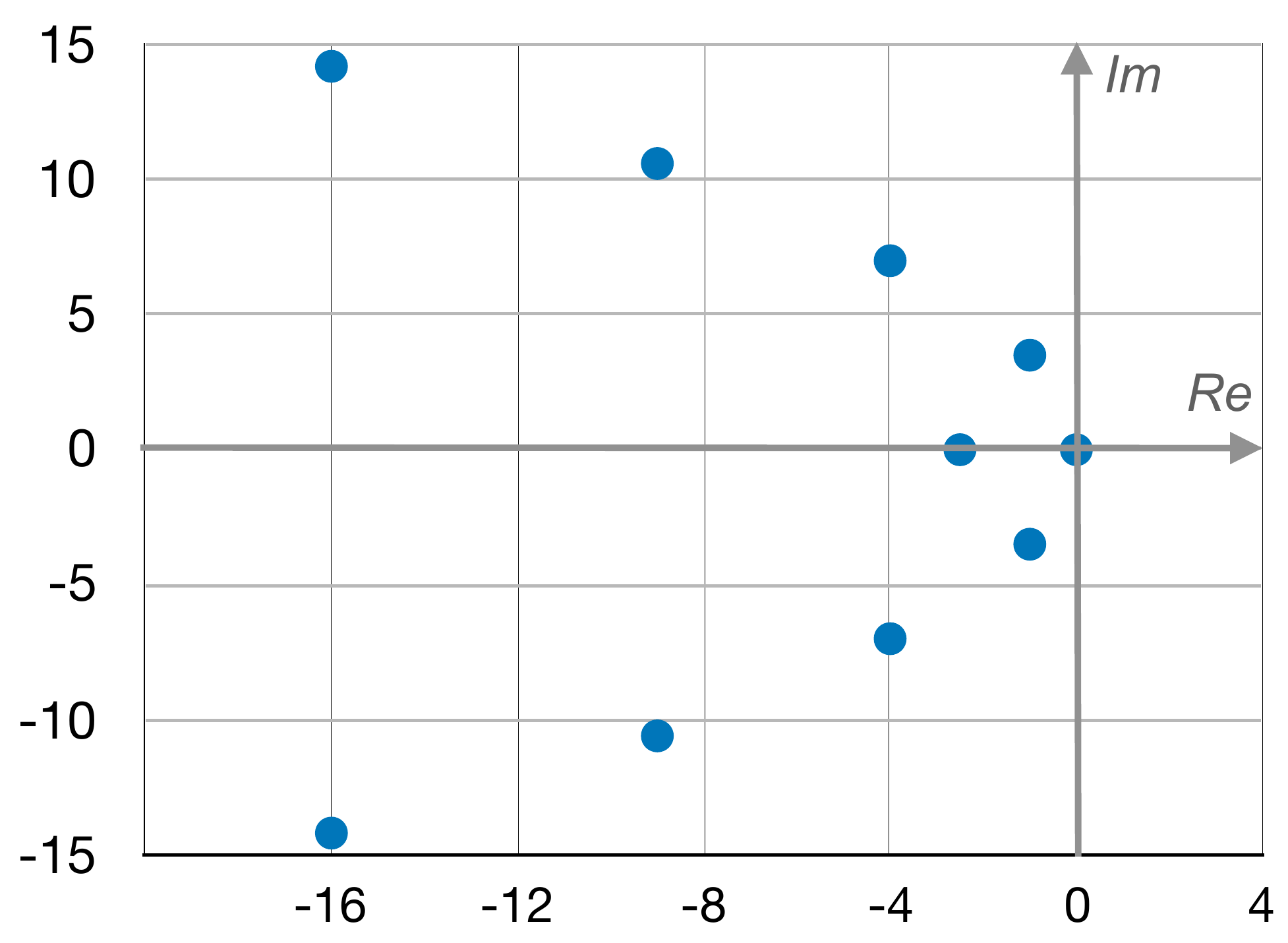}
  \centering
  \caption{The spectrum of the open loop system with viscosity $\eta = 0.1$.}
  \label{spectre-open1}
\end{figure}

In the next section, we shall show that the system can be stabilized by output feedback when the (unknown) viscosity is small and even almost negligible.

\section{Output feedback stabilization of the viscous system} \label{sectionclosedloop}

We now assume that the open loop control system \eqref{unsys} is closed with a dead beat output feedback controller
\begin{equation} \label{dbcontrol}
U(t) = - Y(t- 1).	
\end{equation}
Note that this corresponds to the special case of the controller \eqref{dyncontrol} where $k_1=0$, $k_2=1$ and $\tau = 1/\upsilon = 1$. We know from Section \ref{sectioncontrolproblem} that, in this case, the inviscid system is exponentially stable but the stability is not robust to small perturbations in the propagation speeds. We shall show here that the exponential stability remains in the viscous system and we shall check in Section \ref{sectionrobustness} that, in addition, the exponential stability is robust with small variations in the propagation speeds.
In the frequency domain, the closed loop system \eqref{unsys}, \eqref{dbcontrol} is then:
\begin{subequations} \label{funsys}
\begin{gather}
\tmy_{\!1}(s,1) = f_\eta(s) \tmy_{\!1}(s,0), \\[0.5em]
\tmy_2(s,1) = f_\eta(s) \tmy_2(s,0), \\[0.5em]
\hat \tmy_2(s,1) = e^{-s} \hat \tmy_2(s,0) \label{bcfunsys0}\\[0.5em]
\tmy_{\!1}(s,0) = \tmy_2(s,1) - \hat \tmy_2(s,1), \label{bcfunsys1}\\[0.5em]
\tmy_2(s,0) = \tmy_{\!1}(s,1), \label{bcfunsys2}\\[0.5em]
\hat \tmy_2(s,0) = \tmy_{\!1}(s,1), \label{bcfunsys3}
\end{gather}
\end{subequations}
From these equations, it follows that the characteristic equation of the closed loop system is:
\begin{equation} \label{charac}
f_\eta^2(s) - f_\eta(s)e^{-s} - 1 = 0.	
\end{equation}

Our purpose is now to address the stability of this closed loop system. For a given value of the viscosity $\eta$, the spectrum $\mc{S}_\eta$ of the closed loop system is the set of the poles which are the roots of the characteristic equation \eqref{charac}:
\begin{equation}
	\mc{S}_\eta = \{ s : f_\eta^2(s) - f_\eta(s)e^{-s} - 1 = 0\}.
\end{equation}
Moreover, the maximal spectral abscissa is defined as the supremum of the real parts of the spectrum and denoted as follows:
\begin{equation}
	\sigma_\eta = \sup\{\rp(s) : s \in \mc{S}_\eta \}.
\end{equation}

As a matter of illustration, we present in Figure \ref{spectre-closed} the spectrum of the closed loop system for $\eta = 0.1$.
\begin{figure}[t]
  \includegraphics[width=0.60\textwidth]{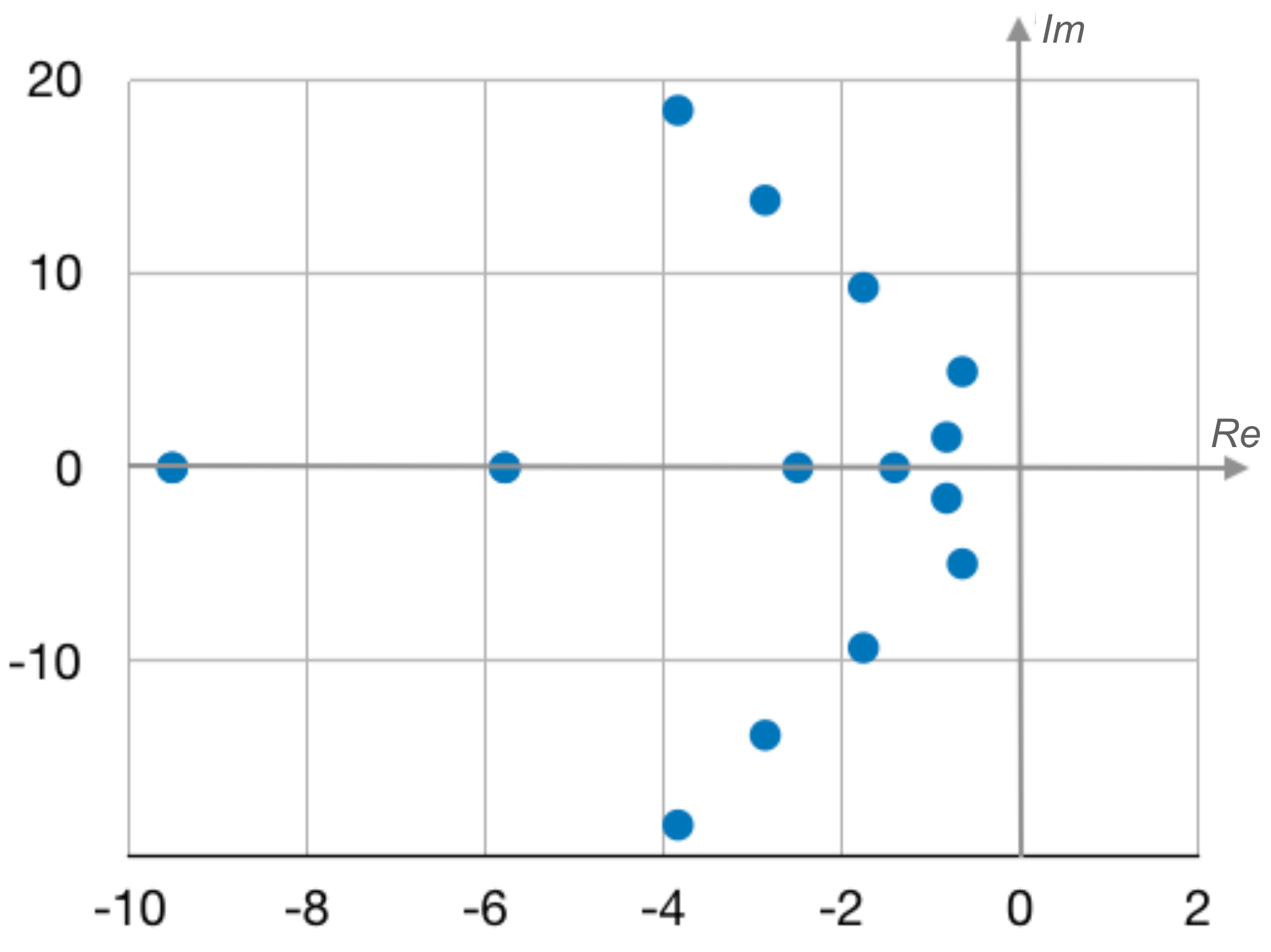}
  \centering
  \caption{The spectrum $\mc{S}_\eta$ of the closed loop system for $\eta = 0.1$.}
  \label{spectre-closed}
\end{figure}

The stability of the closed loop system \eqref{funsys} can be deduced using the spectral mapping theorem. See \cite[Chapter 9, Theorem 3.5]{HalVer93} and \cite{Lic08} in the case $\eta=0$, and reference \cite{BasCorHay22a} for the case $\eta \neq 0$. In particular the system is exponentially stable if (and only if) $\sigma_\eta < 0$.

Then, one of the main results of this paper is given in the following stability theorem.

\begin{theorem}
\label{thmln2}
For any $\delta > 0$, there exists $\eta_{1}>0$ such that, for all $\eta\in(0,\eta_{1})$, the maximal spectral abscissa satisfies
\begin{equation}
\label{sigmaetapetit}
\sigma_\eta \ls -\ln(2) + \delta.
\end{equation}
\qed
\end{theorem}
\begin{conjecture}
\label{conj1}
The bound $\ln(2)$ on the decay rate is optimal. More precisely, for any $\delta > 0$ and $\eta_{1}>0$ there exists $\eta\in(0,\eta_{1})$ such that the maximal abscissa satisfies
\begin{equation}
\label{sigmaetapetit2}
-\ln(2) - \delta \ls \sigma_\eta \ls -\ln(2) + \delta.
\end{equation}
\qed
\end{conjecture}
\begin{remark}[Loss of continuity of the spectral abscissa]
Conjecture \ref{conj1} seems to be verified when looking at the spectrum numerically (see Figure \ref{spectre-closed2}). This would imply a loss of continuity in the sense that any decay rate is achievable when $\eta=0$. However there is a bound $\ln(2)$ as soon as $\eta>0$, even if $\eta$ is arbitrarily small.
\end{remark}

\noi In order to prove Theorem \ref{thmln2}, it is useful to define the complex variable
\begin{equation} \label{defz}
 z = \sqrt{1 + 4\eta s},	
\end{equation}
where we recall that $\sqrt{\cdot}$ is the principal value of the square root. The functions $\lambda_1$ and $\lambda_2$ become
\begin{equation}
\lambda_1 = \dfrac{(1 + z)}{2\eta} \hd \text{and} \hd \lambda_2 = \dfrac{(1 - z)}{2\eta}.
\end{equation}
Then, defining the function
\begin{equation} \label{defX}
X_\eta(z) =  \dfrac{1+z}{2} e^{-(1-z)/2\eta} - \dfrac{1-z}{2} e^{-(1+z)/2\eta},	
\end{equation}
we get, using \eqref{transfer}, that the characteristic equation \eqref{charac} is equivalent to
\begin{equation} \label{charac2}
X^2_\eta(z) + ze^{-(z^2 - 1)/4\eta}X_\eta(z) - z^2 = 0.	
\end{equation}
Let us now observe that $X_\eta(0) = 0$. Consequently $z = 0$ is a root of equation \eqref{charac2} for all $\eta$. This implies obviously that $s = -1/4\eta \in \mc{S}_\eta$  is a pole of the system (i.e. a root of the characteristic equation \eqref{charac}). Note that this pole tends to $-\infty$ as $\eta\rightarrow 0^{+}$ so this is a very stable pole for small (positive) $\eta$.

Let $\mc{Z}_\eta$ denote the set of \textit{non-zero} roots of equation \eqref{charac2}.
Then, using definition \eqref{defX} and solving equation \eqref{charac2} with respect to $X_\eta$, we have that every $z \in \mc{Z}_\eta$ satisfies the equation
\begin{equation} \label{eq22}
\dfrac{(1+z)}{z} e^{-(1-z)/2\eta} - \dfrac{(1-z)}{z} e^{-(1+z)/2\eta}  = -e^{-(z^2-1)/4\eta} \pm \sqrt{e^{-(z^2-1)/2\eta} + 4}.
\end{equation}
We now consider a sequence
\begin{equation} \label{sequence-eta}
	(\eta_n)_{n \in \N} \hd \text{with} \hd 0 < \eta_n \in \real, \; \forall n \in \N \hd \text{and} \hd \lim_{n \rightarrow +\infty} \eta_n = 0^+,
\end{equation}
and an associated sequence
\begin{equation} \label{sequence-s}
	(s_n)_{n \in \N} \hd \text{such that} \hd s_n \in \mc{S}_{\eta_n}, \forall n \in \N.
\end{equation}

In order to prove Theorem \ref{thmln2}, we will look to the adherent points of the sequences $(s_n)_{n \in \N}$ when $n \rightarrow +\infty$ (i.e. when $\eta_n \rightarrow 0^+$). By definition, we know that
$\bar{s}$ is an adherent point of a sequence $(s_n)_{n \in \N}$ if and only if there exists a subsequence which converges to $ \bar s$. With a slight abuse of notation, we will write
\begin{equation}
	s_n \longrightarrow \bar s \hd \text{or} \hd \lim_{n \rightarrow +\infty} s_n = \bar s
\end{equation}
to signify that $\bar s$ is an adherent point of a sequence $(s_n)_{n \in \N}$ but it is implied that the convergence in fact only relies on the adequate subsequence. This holds also for all other sequences that are introduced later in this article.

The proof of Theorem \ref{thmln2} is built from the two following lemmas.

\begin{lemma} \label{lemma1} Let $(\eta_n)_{n \in \N}$ be a sequence of the form \eqref{sequence-eta} and $(z_n)_{n \in \N}$ be any associated sequence of elements of $\mathcal{Z}_{\eta}$.
Let $\bar z$ be an adherent point of the sequence $(z_n)_{n \in \N}$. Then
	\begin{align}
	&\text{a)} \;\; \rp (\bar z^2) \ls 1 \hd \big(\text{precisely: } \rp (\bar z^2) \in [-\infty, 1]\; \big);\\[0.5em]
	&\text{b)} \;\; \text{if } \rp (\bar z^2) = 1 \text{ then } \left(\rp (\bar z) \right)^{\! 2} = 1 \text{ and } \ip (\bar z) = 0.
\end{align}
\end{lemma}
\begin{proof}
	The proof of this lemma is given in Appendix~\ref{app_proof-lemma-1}.
\end{proof}

\begin{lemma} \label{lemma2} Let $(\eta_n)_{n \in \N}$ be a sequence of the form
\eqref{sequence-eta} and $(s_n)_{n \in \N}$ be any associated sequence of the form \eqref{sequence-s}.
Let $\bar s$ be an adherent point of $(s_n)_{n \in \N}$. Then
\begin{equation}
\rp (\bar s) \ls - \ln (2) \hd \big(\text{precisely: } \rp (\bar s) \in [-\infty, -\ln(2)]\; \big).	
\end{equation}
\end{lemma}
\begin{proof}
Let $\bar s$ be an adherence point of the considered sequence $(s_n)_{n \in \N}$. We restrict to a subsequence (still denoted $(s_n)_{n \in \N}$) such that $(s_n)_{n \in \N}$ converges to $\bar s$.
Consider the sequence
\begin{equation} \label{sequence-z}
	(z_n)_{n \in \N} \hd \text{such that} \hd z_n = \sqrt{1+4\eta s_{n}}, \forall n \in \N.
\end{equation}
By definition $z_{n}\in\mathcal{Z}_{\eta}$ for any $n\in\N$. Let $\bar z$ be an adherence value of this sequence $(z_{n})_{n\in\mathbb{N}}$ and let us again restrict to a subsequence (still denoted $(z_{n})_{n\in\mathbb{N}}$) that converges to $\bar z$. Note that since it is a subsequence, we still have $s_{n}\rightarrow \bar s$. From Lemma \ref{lemma1}, we know that necessarily $\rp (\bar z^2) \in [-\infty, 1]$.

Let us first assume that $\rp (\bar z^2)$ is strictly smaller than 1, i.e. $\rp (\bar z^2) \in [-\infty, 1)$. In that case, from the definition of $s_n$ in \eqref{sequence-s}, we have
\begin{equation} \label{formula}
	\rp (\bar s) = \lim_{n \rightarrow +\infty} \rp (s_n) = \lim_{n \rightarrow +\infty} \dfrac{\rp (z_n^2) - 1}{4 \eta_n} = \dfrac{\lim_{n \rightarrow +\infty} \rp (z_n^2) - 1}{\lim_{n \rightarrow +\infty} 4 \eta_n} = \dfrac{\rp (\bar z^2) - 1}{\lim_{n \rightarrow +\infty} 4 \eta_n}.
\end{equation}
Since, by the definition \eqref{sequence-eta}, we know that $\lim_{n \rightarrow +\infty} 4 \eta_n = 0^+$, we can conclude
\begin{equation}
	\rp (\bar s) = - \infty
\end{equation}
and the lemma is proved.

Let us now assume that $\rp (\bar z^2) = 1$. From Lemma \ref{lemma1} we know that, if $\rp (\bar z^2) = 1$, then necessarily $\rp (\bar z) = \pm 1$ and $\ip (\bar z) = 0$; hence $\bar{z} = \pm1$.  We address the case where $\bar z = 1$ (the reader can easily handle the case $\bar z = -1$ by symmetry) and we introduce the notations
\begin{equation} \label{notation-yn}
	y_n = z_n -1, \hd a_n = \rp(y_n), \hd b_n = \ip (y_n).
\end{equation}
With these notations we have
\begin{equation} \label{defsn}
	\rp (s_{n}) = \dfrac{\rp (z_{n}^{2})-1}{4\eta_n} = \dfrac{\rp ((y_n+1)^{2})-1}{4\eta_n}
= \dfrac{a^2_{n}-b_n^{2}}{4\eta_n} + \dfrac{a_n}{2\eta_{n}}.
\end{equation}
Here the determination of adherent points of the sequence $(\rp (s_n))_{n \in \N}$, induced by the limit $\rp (z_n) \rightarrow 1$ is clearly more delicate because it cannot be directly derived from formula \eqref{defsn} and requires a development which is detailed hereafter.

With the notation \eqref{notation-yn}, the function $X_\eta$ defined in \eqref{defX} becomes
\begin{equation}
	X_{\eta} = (1+\frac{y_{n}}{2})e^{y_{n}/2\eta_n}+ \frac{y_{n}}{2}e^{-1/\eta_n-y_{n}/2\eta_{n}},
\end{equation}
such that the characteristic equation \eqref{charac2} is written
\begin{align}
&((1+\frac{y_{n}}{2})e^{y_{n}/2\eta_{n}}+ \frac{y_{n}}{2}e^{-1/\eta_{n}-y_{n}/2\eta_{n}})^{2} \nonumber\\
&\hh \hh +(1+y_{n})e^{-y_{n}/2\eta_{n}-y_{n}^{2}/4\eta_n}[(1+\frac{y_{n}}{2})e^{y_{n}/2\eta_{n}}+ \frac{y_{n}}{2}e^{-1/4\eta_{n}-y_{n}/2\eta_{n}}] = (1+y_{n})^{2}.	
\end{align}
Now, because $\eta_n \rightarrow 0^+$ and $y_n \rightarrow 0$ (since $z_{n}\rightarrow \bar{z}$), we have
\begin{align}
&(1+\frac{y_{n}}{2})^{2}e^{y_{n}/\eta_{n}}+ o(e^{-1/\eta_{n}})\nonumber \\
& \hh \hh +(1+y_{n})(1+\frac{y_{n}}{2})e^{-y_{n}^{2}/4\eta_{n}}+o(e^{-1/4\eta_{n}-y_{n}/\eta_{n}-y_{n}^{2}/4\eta_{n}}) = (1+y_{n})^{2}.	
\end{align}
This implies that, for $n \rightarrow +\infty$,
\begin{equation} \label{limit}
e^{y_{n}/\eta_{n}}+\frac{1+y_{n}}{1+\frac{y_{n}}{2}}e^{-y_{n}^{2}/4\eta_{n}} \longrightarrow 1.
\end{equation}
With \eqref{defsn}, \eqref{limit} becomes
\begin{equation} \label{eq:limyab}
e^{a_n/\eta_{n}}e^{ib_n/\eta_{n}}+\frac{(1+a_n+ib_n)(1+a_n/2-ib_n/2)}{|1+\frac{y_{n}}{2}|^{2}}e^{-a_n^{2}/4\eta_n}e^{b_n^{2}/4\eta_n}e^{-ia_nb_n/2\eta_{n}} \rightarrow 1,
\end{equation}
or, separating the real and imaginary parts,
\begin{gather}
e^{a_n/\eta_{n}}\cos({\textstyle \frac{b_n}{\eta_{n}}}) + e^{-a_n^{2}/4\eta_{n}}e^{b_n^{2}/4\eta_{n}}\left[\left(1+{\textstyle \frac{3a_n +a_n^2 + b_n^2}{2}}\right)\cos({\textstyle \frac{a_nb_n}{2\eta_{n}}})+{\textstyle \frac{b_n}{2}}\sin({\textstyle \frac{a_n b_n}{2\eta_{n}}})\right]\frac{1}{|1+\frac{y_{n}}{2}|^{2}}\rightarrow 1, \label{eq:limyab-real}\\[0.5em]
e^{a_n/\eta_{n}}\sin({\textstyle \frac{b_n}{\eta_{n}}}) + e^{-a_n^{2}/4\eta_{n}}e^{b_n^{2}/4\eta_{n}}\left[-\left(1+{\textstyle \frac{3a_n +a_n^2 + b_n^2}{2}}\right)\sin({\textstyle \frac{a_nb_n}{2\eta_{n}}})+{\textstyle \frac{b_n}{2}}\cos({\textstyle \frac{a_n b_n}{2\eta_{n}}})\right]\frac{1}{|1+\frac{y_{n}}{2}|^{2}}\rightarrow 0. \label{eq:limyab-imag}
\end{gather}
Let us now denote $\bar a$ an adherence point of the sequence $(a_n/2\eta_n)_{n \in \N}$.  We shall discuss successively the three cases : $\bar a = -\infty$, $\bar a \in (-\infty, +\infty)$ and $\bar a = +\infty$.\\

\noi \textbf{The case} $\mathbf{\bar a = - \infty.}$ \\
In this case the lemma is trivial because, since $a_n \rightarrow 0$ and $a_n/2\eta_n \rightarrow \bar a = -\infty$, then $a_n^2/4\eta_n + a_n/2\eta_n \rightarrow -\infty$ and therefore, from \eqref{defsn}, $\rp (\bar s) = - \infty$.\\

\noi \textbf{The case} $\mathbf{\bar a \in (- \infty, +\infty).}$

In this case $a_n^2/2\eta_n \rightarrow 0$ and $a_n b_n/2\eta_n \rightarrow  0$ because $a_n + i b_n = y_n \rightarrow 0$ while $a_n/2\eta_n \rightarrow \bar a\in\mathbb{R}$. Then from \eqref{eq:limyab-real}, we have
\begin{equation} \label{eq:limyab1}
e^{a_n/\eta_{n}}\cos({\textstyle \frac{b_n}{\eta_{n}}}) + e^{b_n^{2}/4\eta_{n}}(1 + o(1)) \longrightarrow 1.
\end{equation}
Let us now denote $\bar c \in [0,+\infty]$ an adherence point of the sequence $(b_n^{2}/4\eta_{n})_{n \in \N}$ and $\bar \kappa \in [-1,1]$ an adherent point of the sequence $(\cos({\textstyle \frac{b_n}{\eta_{n}}}))_{n \in \N}$. Then, taking the limit\footnote{This can be done by restricting to a subsequence where both $(b_n^{2}/4\eta_{n})_{n \in \N}$ and $(\cos({\textstyle \frac{b_n}{\eta_{n}}}))_{n \in \N}$ converge, using a diagonal argument.}, we have from \eqref{eq:limyab1}
\begin{equation} \label{e2akec}
e^{2 \bar a} \bar \kappa + e^{\bar c} = 1.	
\end{equation}
Because $e^{\bar a}$ is bounded and $e^{\bar a} > 0$, and $b_{n}^{2}/4\eta_{n}\geq 0$ for any $n\in\mathbb{N}$, this implies necessarily that
\begin{equation}
	\bar c \in [0, +\infty)\hd\text{and}\hd e^{\bar c} \gs 1 \hd \text{and} \hd \bar \kappa \in [-1, 0].
\end{equation}
Then it means in particular that $b_n^2/\eta_n$ is bounded when $\eta_n \rightarrow 0^+$, and from \eqref{eq:limyab-imag} we deduce that
\begin{equation}
	e^{a_n/\eta_{n}}\sin({\textstyle \frac{b_n}{\eta_{n}}}) + o(1) \longrightarrow 0
\end{equation}
which implies that $\sin({\textstyle \frac{b_n}{\eta_{n}}}) \rightarrow 0$ and therefore that the only possible adherence points for $\cos({\textstyle \frac{b_n}{\eta_{n}}})$ are $\bar \kappa = \pm 1$. Since $\bar \kappa\in [-1,0]$, we deduce that $\bar\kappa = -1$. Therefore, from \eqref{e2akec},
  \begin{equation} \label{e2aec}
 e^{\bar c} = 1 + e^{2 \bar a}.	
\end{equation}
Let us now consider the function $e^{-\rp (s_n)}$ with $\rp (s_n)$ given by \eqref{defsn}:
\begin{equation}
e^{-\rp (s_n)} = e^{(b_n^2 -a_n^2)/4\eta_n - a_n/2\eta_n} \longrightarrow  e^{\bar c}e^{- \bar a}	= e^{-\bar a} + e^{\bar a} = 2 \cosh(\bar a) \gs 2.
\end{equation}
It follows directly that
\begin{equation}
\rp (\bar s) \ls - \ln (2)	
\end{equation}
and the lemma is proved.\\

\noi \textbf{The case} $\mathbf{\bar a = + \infty.}$

In this case we first observe that $e^{(a_n^2 - b_n^2)/4\eta_n - a_n/2\eta_n} \rightarrow 0$. Then multiplying both sides of \eqref{eq:limyab-real} by this quantity, we get:
\begin{align}
&e^{(a_n^2 - b_n^2)/4\eta_n + a_n/2\eta_n}\cos({\textstyle \frac{b_n}{\eta_n}})+e^{-a_n/2\eta_n}\left[(1+{\textstyle \frac{3a_n + a_n^2 + b_n^2}{2}})\cos({\textstyle \frac{a_nb_n}{2\eta_{n}}})+{\textstyle \frac{b_{n}}{2}}\sin({\textstyle \frac{a_nb_n}{2\eta_{n}}})\right]\frac{1}{|1+\frac{y_{n}}{2}|^{2}} \nonumber\\
&\hspace{8cm}-e^{(a_n^2 - b_n^2)/4\eta_n - a_n/2\eta_n}\rightarrow 0, \label{eq:limyab2}
\end{align}
From this expression, we deduce that
\begin{equation} \label{eq:limyab21}
	e^{(a_n^2 - b_n^2)/4\eta_n + a_n/2\eta_n}\cos({\textstyle \frac{b_n}{\eta_n}}) \longrightarrow 0.
\end{equation}
A similar manipulation of \eqref{eq:limyab-imag} gives
\begin{equation} \label{eq:limyab22}
	e^{(a_n^2 - b_n^2)/4\eta_n + a_n/2\eta_n}\sin({\textstyle \frac{b_n}{\eta_n}}) \longrightarrow 0.
\end{equation}
Combining \eqref{eq:limyab21} and \eqref{eq:limyab22} we obtain
\begin{equation}
	e^{(a_n^2 - b_n^2)/2\eta_n + a_n/\eta_n} = e^{2\rp (s_n)} \longrightarrow 0
\end{equation}
which implies that $\rp (s_n) \rightarrow -\infty$ and the lemma is proved.
\end{proof}

Based on Lemma \ref{lemma2}, we can now give the following proof of Theorem \ref{thmln2}.\\

\noi \textbf{Proof of Theorem \ref{thmln2}.}\\
We assume by contradiction that the theorem does not hold:
\begin{equation}
	\text{For any }\delta > 0, \nexists \, \eta_1 > 0 \text{ such that } \sigma_\eta \ls -\ln (2) + \delta \;\; \forall \eta \in (0, \eta_1).
\end{equation}
This implies that
\begin{equation} \label{forall}
\text{For all } \eta > 0, \; \exists \eta_1\in(0,\eta) \text{ and } \exists s \in \mc{S}_{\eta_1} \text{ such that } \rp (s) > -\ln (2) + \delta.	
\end{equation}
It follows that we can build a sequence $(\eta_n)_{n \in \N}$ of the form \eqref{sequence-eta} and an associated sequence $(s_n)_{n \in \N}$ with $s_n \in \mc{S}_{\eta_n}$. For this sequence we deduce from \eqref{forall} that necessarily $\rp (\bar s) > - \ln (2) + \delta$, which is in contradiction with Lemma \ref{lemma2}. \qed

\begin{figure}[ht]
  \includegraphics[width=0.60\textwidth]{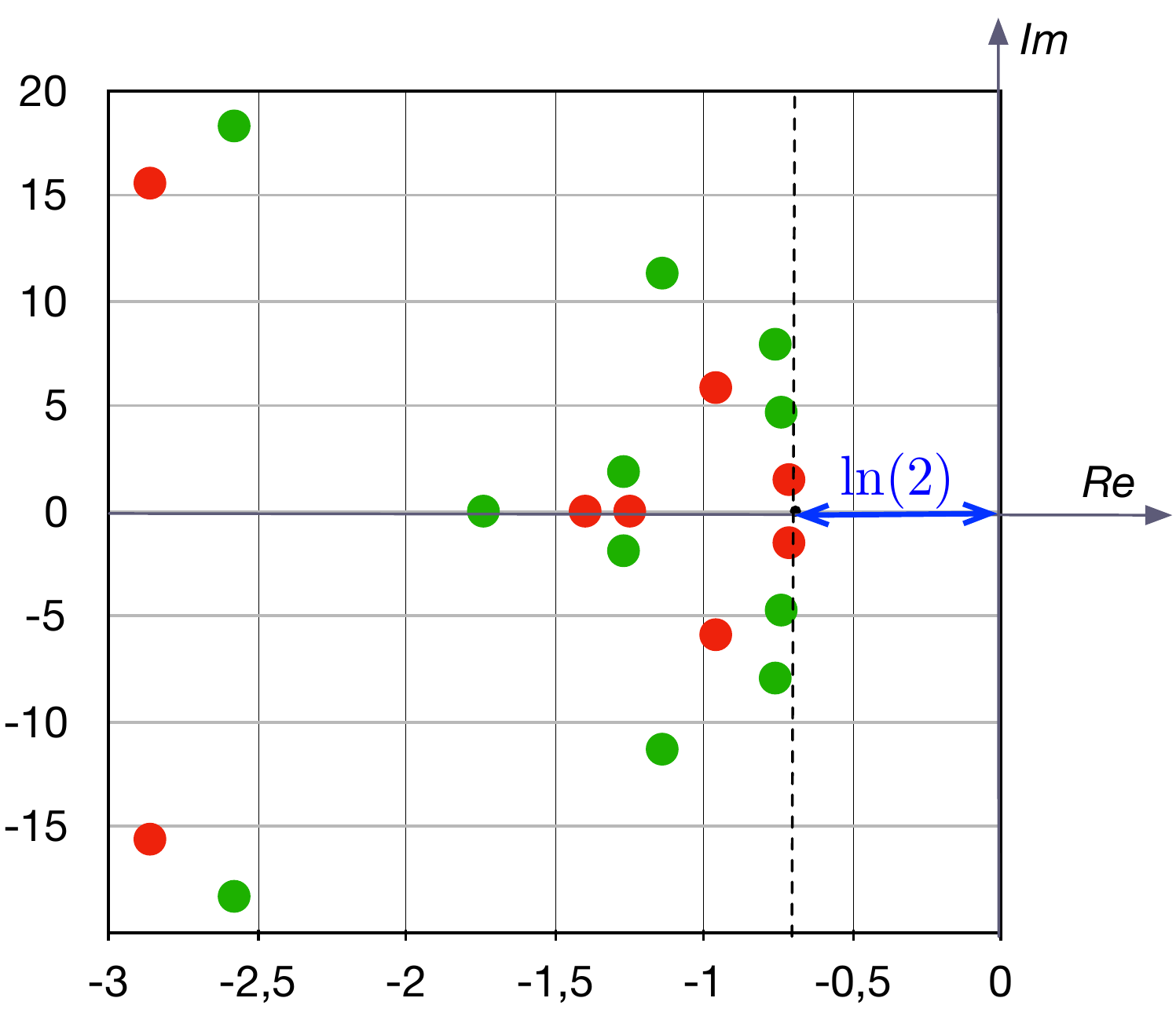}
  \centering
  \caption{The spectrum $\mc{S}_\eta$ of the closed loop system for $\eta = 0.02$ \red{•} and $\eta = 0.2$ \textcolor{fgreen}{•}.}
  \label{spectre-closed2}
\end{figure}

The theorem is illustrated in Figure \ref{spectre-closed2} where the spectrum is represented for the values $\eta$ = 0.02 and 0.2, and where the effectiveness of the stability margin $-\ln(2)$ can be appreciated.

\section{Robustness analysis} \label{sectionrobustness}

We consider again the control problem of the previous section. However, because we want to explicitly account for the sensitivity to delay uncertainties, we introduce an additional perturbation $\varepsilon$ to the nominal transport velocity $\upsilon = 1$. The open loop system is therefore written in the time domain as follows:
\begin{subequations} \label{epsys}
\begin{gather}
\p_t y_1(t,x) + ( 1 + \varepsilon) \p_x y_1(t,x) - \eta \p^2_{xx} y_1(t,x) = 0, \label{epsys1}\\[0.5em]
\p_t y_2(t,x) + ( 1 + \varepsilon) \p_x y_2(t,x) - \eta \p^2_{xx} y_2(t,x) = 0, \label{epsys2}\\[0.5em]
y_1(t,0) = y_2(t,1) + U(t), \label{bcepsys1}\\[0.5em]
y_2(t,0) = y_1(t,1), \label{bcepsys2}\\[0.5em]
\p_x y_1(t,1) = \p_x y_2(t,1) = 0 \label{bcepsys3}\\[0.5em]
Y(t) = y_1(t,1), \label{bcepsys4}
\end{gather}
\end{subequations}
We suppose, as before, that the system is closed with the dead beat output feedback controller
\begin{equation} \label{dbcontroleps}
U(t) = - Y(t-1).	
\end{equation}
Remark that this control law depends on the theoretical delay ($\tau = 1$), ignoring the uncertainty represented by $\varepsilon$.

Remark also that for simplicity we assume here that we have the same uncertainty $\varepsilon$ on both physical subsystems represented by transport equations \eqref{epsys1} and \eqref{epsys2}. This may seem like a simplification but, actually, it can be shown that this single perturbation, even if it is arbitrarily small,
is sufficient to destroy the closed loop stability when $\eta = 0$ (see Figure \ref{spectre-closed4} hereafter).

In this case the transfer functions of the two transport subsystems become
\begin{equation} \label{transfer-2}
f_{\eta,\varepsilon}(s)  = \dfrac{\lambda_1(s) - \lambda_2(s)}{\lambda_1(s) e^{-\lambda_2(s)} - \lambda_2(s) e^{-\lambda_1(s)}}.
\end{equation}
with the function $f_{\eta,\varepsilon}$ now depending on both $\eta$ and $\varepsilon$ because the functions $\lambda_1$ and $\lambda_2$ are modified as follows:
\begin{equation}
\lambda_1(s) = \dfrac{( 1 + \varepsilon) +\sqrt{( 1 + \varepsilon)^2 +4\eta s}}{2\eta}, \hd \lambda_2(s) = \dfrac{( 1 + \varepsilon) -\sqrt{( 1 + \varepsilon)^2 + 4\eta s}}{2\eta}.
\end{equation}
With this definition, the characteristic equation of the closed loop system is now
\begin{equation}
\label{charac4}
f_{\eta,\varepsilon}^2(s) - f_{\eta,\varepsilon}(s)e^{-s} - 1 = 0.
\end{equation}
As in the previous section, we introduce the spectrum $\mathcal{S}_{\eta,\varepsilon}$ and the maximal abscissa $\sigma_{\eta,\varepsilon}$ defined by
\begin{gather}
\mc{S}_{\eta,\varepsilon}=\{s\in\mathbb{C} : s\text{ is solution of } \eqref{charac4}\},\\[0.5em]
\sigma_{\eta,\varepsilon} = \sup\{\rp(s) : s \in \mc{S}_{\eta,\varepsilon} \}.
\end{gather}
We remark that, by definition, we have
\begin{equation}
	\mc{S}_{\eta,0} = \mc{S}_\eta \hd \text{ and } \hd \sigma_{\eta,0} = \sigma_\eta.
\end{equation}

We then have the following robustness theorem.
\begin{theorem}
\label{th2}
Let  $\delta>0$ and $\eta>0$ be such that Theorem \ref{thmln2} holds, i.e.
\begin{equation}
	\sigma_{\eta,0}\ls - \text{ln}(2) + \delta.
\end{equation}
Then there exists $\varepsilon_{1}>0$ such that for any
$\varepsilon\in(-\varepsilon_{1},\varepsilon_{1})$ the maximal spectral
abscissa $\sigma_{\eta,\varepsilon}$ satisfies
\begin{equation}
\label{eq:estimth2}
\sigma_{\eta,\varepsilon}\ls - \text{ln}(2)+2\delta
\end{equation}
\end{theorem}

\begin{proof}
We proceed by contradiction and we assume that Theorem \ref{th2} does not hold. This implies that for any $\varepsilon_{1}>0$ there exists $\varepsilon\in (-\varepsilon_{1}, \varepsilon_{1})$ such that
\begin{equation}
\sigma_{\eta,\varepsilon}> - \text{ln}(2)+2\delta,
\end{equation}
and therefore that there exists $s\in \mathcal{S}_{\eta,\varepsilon}$ such that
\begin{equation}
\label{contsigma}
\rp(s)> - \text{ln}(2)+2\delta.
\end{equation}
Hence we can define a sequence
\begin{equation} \label{sequence-eta-2}
	(\varepsilon_n)_{n \in \N} \hd \text{with}  \hd \lim_{n \rightarrow +\infty} \varepsilon_n = 0,
\end{equation}
and an associated sequence $(s_n)_{n \in \N}$ such that
\begin{equation} \label{sequence-s-2}
	\rp(s_n) > - \text{ln}(2)+2\delta, \;\; \text{with} \;\; s_n\in \mathcal{S}_{\eta,\varepsilon_n} \; \forall n \in \N.
\end{equation}
\begin{itemize}
\item If the sequence $(s_{n})_{n\in\mathbb{N}}$ is bounded, then we can extract a subsequence that converges to a limit $\bar{s}\in\mathbb{C}$. We still denote this subsequence by $(s_{n})_{n\in\mathbb{N}}$. As $\varepsilon_{n}\rightarrow 0$ we can pass to the limit in \eqref{charac4} and deduce that $\bar{s}\in \mathcal{S}_{\eta,0}=\mathcal{S}_{\eta}$ and therefore, from the assumption on $\delta$ and $\eta$, that
\begin{equation}
\label{contsigma2}
\rp(\bar{s})\ls - \text{ln}(2)+\delta.
\end{equation}
On the other hand, again passing to the limit, we deduce from \eqref{sequence-s-2} that
\begin{equation}
\rp(\bar{s})\gs - \text{ln}(2)+2\delta>- \text{ln}(2)+\delta,
\end{equation}
which is in contradiction with \eqref{contsigma2}.
\item If the sequence $(s_{n})_{n\in\mathbb{N}}$ is unbounded, then we can extract a subsequence such that $|s_{n}|\rightarrow +\infty$. We still denote this subsequence by $(s_{n})_{n\in\mathbb{N}}$. We denote $s_{n} = r_{n}e^{i\theta_{n}}$, with $r_{n}\in\mathbb{R}_{+}$ and $\theta_{n}\in[-\pi,\pi)$, for $n\in\mathbb{N}$. Since $r_{n} =|s_{n}|\rightarrow +\infty$ and  \eqref{sequence-s-2} holds, we deduce that for $n$ sufficiently large $\theta_{n}\in (-5\pi/8,5\pi/8)$. Denoting $(1+\varepsilon_{n})+4\eta s_{n} = r_{n}^{1}e^{i\theta_{n}^{1}}$ with $r_{n}^{1}\in\mathbb{R}_{+}$ and $\theta^1_{n}\in[-\pi,\pi)$ we deduce that for $n$ sufficiently large (depending on $\eta$), $r_{n}^{1}\rightarrow +\infty$ and $\theta_{n}^{1}\in(-3\pi/4,3\pi/4)$. Thus
\begin{equation}
\sqrt{(1+\varepsilon_{n})+4\eta s_{n}}= \sqrt{r_{n}^{1}}e^{i\theta_{n}^{1}/2},
\end{equation}
where $\theta_{n}^{1}/2\in (-3\pi/8, 3\pi/8)$, which implies that $\rp(\sqrt{1+4\eta s_{n}})\rightarrow +\infty$. From the definition of $\lambda_{1}$ and $\lambda_{2}$, we deduce that
\begin{equation}
\label{eq:limitlambda}
\rp(\lambda_{1}(s_{n}))\rightarrow +\infty,\;\;\rp(\lambda_{2}(s_{n}))\rightarrow -\infty,
\end{equation}
and
\begin{equation}
f_{\eta,\varepsilon_{n}}=2\left[(\frac{(1+\varepsilon_{n})}{\sqrt{(1+\varepsilon_{n})+4\eta s_{n}}}+1)e^{-\lambda_{2}(s_{n})}+(\frac{(1+\varepsilon_{n})}{\sqrt{(1+\varepsilon_{n})+4\eta s_{n}}}-1)e^{-\lambda_{1}(s_{n})}\right]^{-1}.
\end{equation}
Using \eqref{eq:limitlambda} and observing that $(\frac{(1+\varepsilon_{n})}{\sqrt{(1+\varepsilon_{n})+4\eta s_{n}}}+1)\rightarrow 1$, we obtain that $f_{\eta,\varepsilon_{n}}\rightarrow 0$. Moreover, from \eqref{sequence-s-2}, $e^{-s_{n}}$ is bounded. Hence, we deduce that
\begin{equation}
f_{\eta,\varepsilon_n}^2(s_n) - f_{\eta,\varepsilon_n}(s_n)e^{-s_n} - 1 \longrightarrow -1
\end{equation}
and, from  \eqref{charac4}, we obtain again a contradiction.
\end{itemize}
In both cases we obtain a contradiction, which means that there exists $\varepsilon_{1}>0$ such that \eqref{eq:estimth2} holds for any $\varepsilon\in (-\varepsilon_{1},\varepsilon_{1})$. This concludes the proof of Theorem \ref{th2}.
\end{proof}
\begin{figure}[h]
  \includegraphics[width=0.65\textwidth]{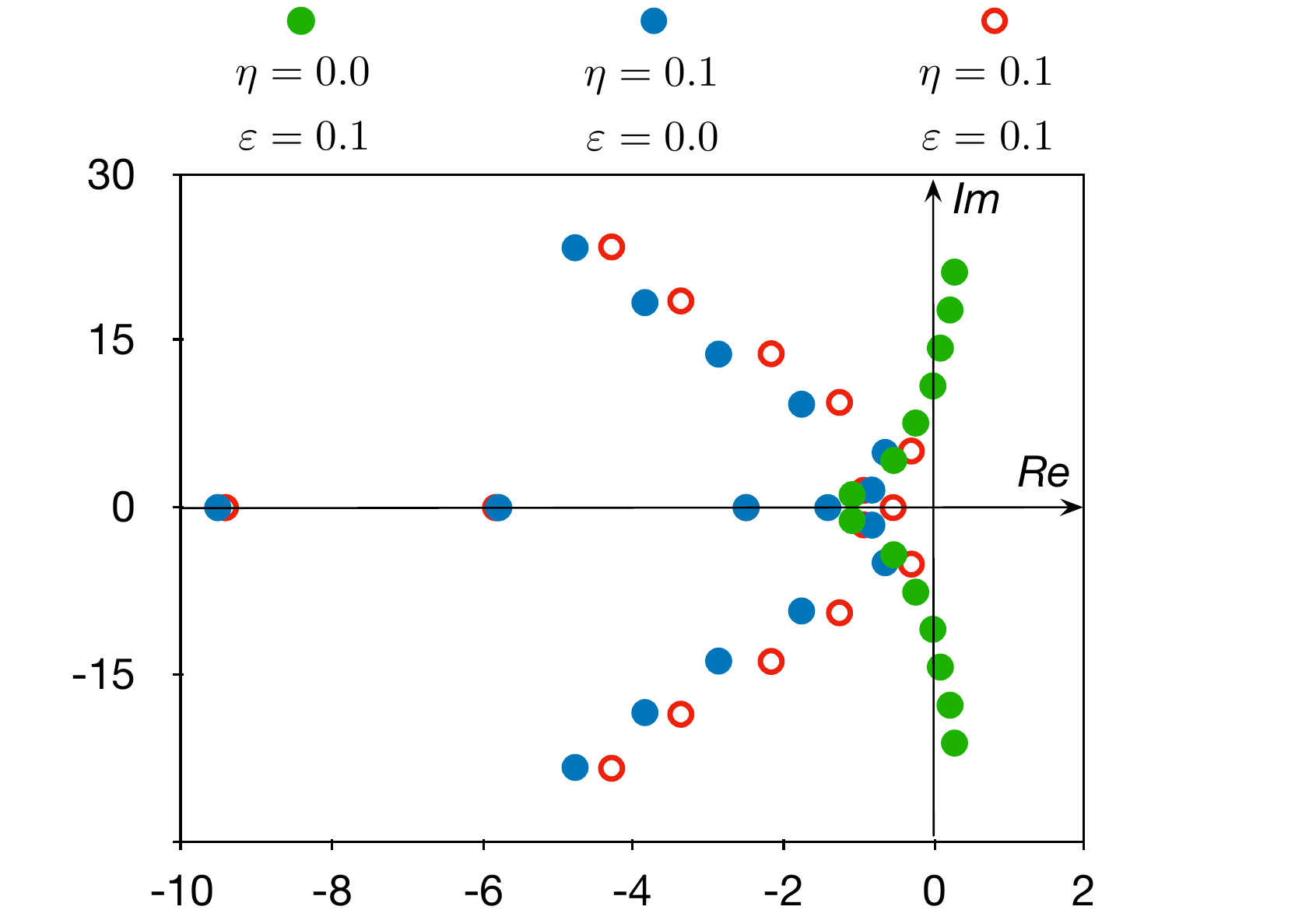}
  \centering
  \caption{The spectrum $\mc{S}_{\eta, \varepsilon}$ of the closed loop system: influence of viscosity on  stability in case of delay uncertainty.}
  \label{spectre-closed4}
\end{figure}
The theorem is illustrated in Figure \ref{spectre-closed4}. In this figure, we can see what happens in the situation where there is no diffusion ($\eta = 0$) but a slight uncertainty ($\varepsilon = 0.1$) of the transport velocity: $\upsilon = 1 + \varepsilon = 1.1$ instead of $\upsilon = 1$. Although the ideal system (without modelling uncertainty) should be exponentially stable, it appears that it becomes unstable with poles (represented by green dots in Figure \ref{spectre-closed4}) moving to the right-half complex plane.

In contrast, when there is some diffusion ($\eta = 0.1$) and no uncertainty ($\varepsilon = 0$), we know from Theorem \ref{thmln2} that the closed loop system must be stable as it can be seen with the spectrum of blue dots (actually reprinted from Figure \ref{spectre-closed}) which is entirely strictly located in the left half plane.

Then, illustrating Theorem \ref{th2}, the robustness of the control in presence of diffusion is clearly evidenced by the spectrum made up of red dots which results from a small shift of the initial blue spectrum but remains entirely in the left half plane.

\section{Comparison with a simpler system} \label{sectionscalar}

In this section, we consider the case of a system which is very close to  the previous one but slightly simpler. Somewhat surprisingly, we will see that in this case the addition of a diffusion term does not seem to strengthen the system stability but instead destroys the stability.

We start with an ideal system without diffusion nor modelling uncertainty which is clearly a simplified form of the physical system \eqref{closinvisys} that we have considered in Section \ref{sectioncontrolproblem} and is written as follows:
\begin{subequations} \label{scalarsys}
\begin{gather}
\p_t y(t,x) + \p_x y(t,x) = 0, \label{scalarsys1}\\[0.5em]
\p_t \hat y(t,x) + \p_x \hat y(t,x) = 0, \label{scalarsys2}\\[0.5em]
\bpm y(t,0) \\[0.5em] \hat y(t,0) \epm = \underbrace{\bpm  1 & - 1 \\[0.5em] 1 & -1  \epm }_{\displaystyle \mbK} \bpm y(t,1) \\[0.5em] \hat y(t,1) \epm.
\end{gather}
\end{subequations}
In the frequency domain, 
 the characteristic equation reduces to $e^{2s} = 0$, meaning that the system is exponentially stable (for any decay rate). In fact one can observe easily that the system is finite time stable: for any time $t\geq 1$, $\hat{y}(t,\cdot)=y(t,\cdot)$ on $[0,1]$, thus from the boundary condition at $x=0$, $y(t,0)=0$, which means that $y(t,\cdot) \equiv 0$ for $t\geq 2$. However, in this case also, the exponential stability is not robust w.r.t. to delay inaccuracy because $\bar \rho(\mbK) = \sqrt{2}$.

On the basis of our previous results in this paper, it seems natural to conjecture that the addition of a diffusion term in equation \eqref{scalarsys1} should allow to strengthen the system stability.  To address this issue, the system dynamics  \eqref{scalarsys} are modified with an additional diffusion parameter $\eta$ as follows:
\begin{subequations} \label{clscalarsys}
\begin{gather}
\p_t y(t,x) + \p_x y(t,x) - \eta \p_{xx} y(t,x)= 0, \label{clscalarsys1}\\[0.5em]
\p_t \hat y(t,x) + \p_x \hat y(t,x) = 0, \label{clscalarsys2}\\[0.5em]
y(t,0) = y(t,1) - \hat y(t,1), \label{bcclscalarsys1}\\[0.5em]
\hat y(t,0) = y(t,0), \label{bcclscalarsys2}\\[0.5em]
\p_x y(t,1) = 0. \label{bcclscalarsys3}
\end{gather}
\end{subequations}
For this system in the frequency domain, after calculations similar to those in Sections \ref{sectionopenloop} and \ref{sectionclosedloop}, it can be shown that the characteristic equation is
\begin{equation}
F_\eta(s) - 1 = 0 \hd \text{with} \hd F_\eta(s) = \dfrac{\left(\lambda_{1}(s)e^{\lambda_{1}(s)}-\lambda_{2}(s)e^{\lambda_{2}(s)}\right)(1+e^{-s})}{(\lambda_{1}(s)-\lambda_{2}(s))e^{\lambda_{1}(s)+\lambda_{2}(s)}},
\end{equation}
where $\lambda_1(s)$ and $\lambda_2(s)$ are given by \eqref{lambda12} and repeated here for convenience:
\begin{equation}
\lambda_1(s) = \dfrac{1 +\sqrt{1 + 4 \eta s}}{2\eta}, \hd \lambda_2(s) = \dfrac{1 -\sqrt{1 + 4 \eta s}}{2\eta}.
\end{equation}
For a given value of $\eta$, as in Section \ref{sectionclosedloop}, we use the following notations for the spectrum and the maximal spectral abscissa:
\begin{gather}
	\mc{S}_\eta = \{ s\in\mathbb{C} : F_\eta(s) - 1 = 0\}, \\[0.5em]
	\sigma_\eta = \sup\{\rp(s) : s \in \mc{S}_\eta \}.
\end{gather}
We then have the surprising observation that a slight diffusion in the system has, in this case, a clear destabilizing effect. This is graphically illustrated in Figure \ref{spectre5} and leads to the following conjecture.

\begin{conjecture} \label{th3}
For all $\epsilon > 0$, there exists $\eta_{1}>0$ such that for all $\eta\in(0,\eta_{1})$ the maximal spectral abscissa  satisfies the inequality $\sigma_\eta > - \epsilon$.
\qed
 \end{conjecture}

\begin{figure}[h]
  \includegraphics[width=0.50\textwidth]{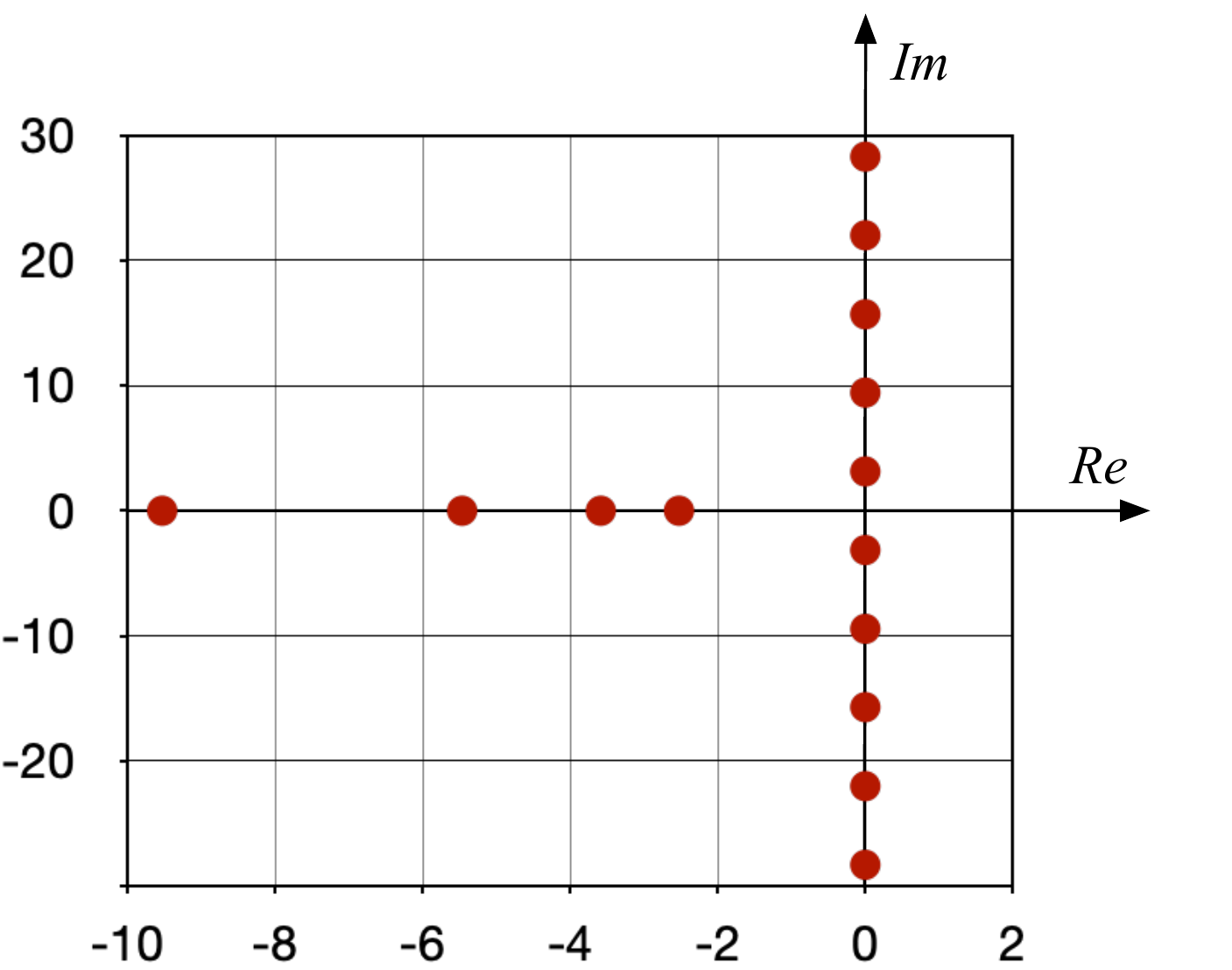}
  \centering
  \caption{The spectrum $\mc{S}_{\eta}$ of the system \eqref{clscalarsys} with $\eta = 0.1$.}
  \label{spectre5}
\end{figure}

\section{Conclusion} \label{sectionconclusion}

We have discussed the output feedback stabilization of an unstable open loop system which is made up of two interconnected transport equations and provided with anti-located boundary sensing and actuation. We have shown that the system can be stabilized by a dynamic controller that involves a delayed output feedback which turns out to be non-robust with respect to delay uncertainties. Then we have shown that the designed control law can however stabilize the system in a robust way when there is a small unknown diffusion in the plant.\\

Our work in progress on this topic \cite{BasCorHay22a} will be focused on the output feedback stabilization of the motion of a viscous fluid represented by $2\times 2$ hyperbolic PDEs when the control input is the flow rate at one boundary while the measurable output is the fluid density at the other boundary. There is, in this case, an important difference which lies in how viscosity affects the model, inducing a distributed internal coupling between the two partial differential equations. This implies that the system can no longer be considered as a feedback interconnection of two independent scalar transport equations which may lead to additional difficulties for the stability analysis.

\section*{Acknowledgement}
The authors would like to thank the Bernoulli center for fundamental studies (CIB) and Joachim Krieger for their hospitality and their support, as well as the projet PEPS JCJC 2022 of CNRS - INSMI.

\appendix
\section{Computation of $\bar \rho(\mbK) \gs 1$.} \label{appendixrho}

In this appendix we consider the matrix $\mbK$ defined in equation \eqref{closinvisysbc} as follows:
\begin{equation}
	\mbK = \bpm -2k_1 & 1 & - k_2 \\[0.5em] 1 & 0 & 0 \\[0.5em]\ 1 & 0 & 0\epm.
\end{equation}
 From \cite[Proposition 3.7]{CorBasdAn08} and \cite[Theorem 3.12]{BasCor14}, we have
\begin{equation}\label{barrhorho2}
  \bar \rho(K)=\rho_2(K)
\end{equation}
with
\begin{align}
\rho_2(\mb{K}) &:= \inf \{ \Vert D\mbK D^{-1} \Vert_2 ; D \text{ is a positive diagonal matrix} \} \\[0.5em]
&= \inf \left\{ \sqrt{\lambda_{\max} \big[ \big(D \mb{K} D^{-1}\big)^T \big(D \mb{K} D^{-1}\big)\big]}; D \text{ is a positive diagonal matrix} \right\}.
\end{align}
With the (normalized) matrix
\begin{equation}
	D = \bpm 1 & 0 & 0\\ 0 & \theta_2 & 0 \\0 & 0 & \theta_3 \epm, \hd \theta_2 > 0, \;\theta_3 > 0,
\end{equation}
we have
\begin{equation}
\mathcal{M} = \big(D \mb{K} D^{-1}\big)^T \big(D \mb{K} D^{-1}\big) = \bpm 4k_1^2 + \theta_2^2 + \theta_3^2 & -2k_1 \theta_2^{-1} & 2 k_1 k_2 \theta_3^{-1} \\[0.5em] -2k_1 \theta_2^{-1} & \theta_2^{-2} & -k_2 \theta_2^{-1} \theta_3^{-1} \\[0.5em] 	2 k_1 k_2 \theta_3^{-1} & -k_2 \theta_2^{-1} \theta_3^{-1} & k_2^2 \theta_3^{-2} \epm.
\end{equation}
Then we have
\begin{equation} \label{polynomial}
\det\big(\lambda \mb{I}_3 - \mathcal{M}\big) = \lambda (\lambda^2 - \beta \lambda + \gamma)
\end{equation}
with
\begin{align}
&\beta = 4k_1^2 + (\theta_2^2 + \theta_2^{-2}) + (\theta_3^2 + k_2^2 \theta_3^{-2}), \label{defbeta} \\[0.5em]
&\gamma = 1 + \theta_2^{-2} \theta_3^2 + k_2^2 + k_2^2 \theta_2^2 \theta_3^{-2}. \label{defgamma}
\end{align}
From \eqref{polynomial} we see that the eigenvalues of the matrix $\mathcal{M}$ are
\begin{equation}
\lambda = 0 \hd \text{and} \hd \lambda = \dfrac{\beta \pm \sqrt{\beta^2 - 4 \gamma}}{2}.	
\end{equation}
From \eqref{defbeta} we know that $\beta \gs 0$. Moreover, $\beta^2 - 4 \gamma \gs 0$ because the matrix $\mathcal{M}$ is symmetric. It follows that
\begin{align}
\rho_2(\mbK) = \inf_{\theta_2, \theta_3} \sqrt{	\dfrac{\beta + \sqrt{\beta^2 - 4 \gamma}}{2}}.
\end{align}
From \eqref{defbeta} and \eqref{defgamma}, after some computations, we get
\begin{multline}
\beta^2 - 4 \gamma = 16\,k_1^4 + 8k_1^2(\theta_2^2 + \theta_2^{-2}) + 8k_1^2(\theta_3^2 + k_2^2 \theta_3^{-2}) + (\theta_2^4 + \theta_2^{-4}) + (\theta_3^4 + k_2^4 \theta_3^{-4})	\\[0.5em] + 2 (\theta_2^2 \theta_3^2 + k_2^2 \theta_2^{-2} \theta_3^{-2}) - 2 - 2k_2^2 - 2 (\theta_2^{-2} \theta_3^2 +  k_2^2\theta_2^2 \theta_3^{-2}).
\end{multline}
From \eqref{defbeta} and this latter expression it can be verified that $\beta + \sqrt{\beta^2 - 4 \gamma}$ is minimal if and only if
\begin{equation}
	\theta_2^2 = 1 \hd \text{and} \hd \theta_3^2 = |k_2|.
\end{equation}
With these values we then have
\begin{align}
	\rho_2(\mb{K}) &= \inf_{\theta_2, \theta_3} \sqrt{	\dfrac{\beta + \sqrt{\beta^2 - 4 \gamma}}{2}} \\[0.5em]
	&= |k_1| + \sqrt{1 + k_1^2 + |k_2|} \;\; \gs 1 \hd \text{for all} \; (k_1, k_2).
\end{align}

\section{Proof of Lemma \ref{lemma1}.}
\label{app_proof-lemma-1}

\noi \textbf{Part a)}

Assume by contradiction that $\rp (\bar z^2) \in (1, +\infty]$. Then
\begin{equation}
	e^{-(\rp (z_n^2) - 1)/4\eta_n} \longrightarrow 0
\end{equation}
and it follows that the right-hand side of \eqref{eq22} converges to $\pm 2$. Then, denoting $a_n = \rp (z_n)$ and $b_n = \ip (z_n)$, \eqref{eq22} implies
\begin{equation}
	\label{contralem21}
\frac{1}{2}\left(\left(\frac{a_n-ib_n}{a_n^{2}+b_n^{2}}+1\right)e^{-(1-a_n-ib_n)/2\eta_n}-
\left(\frac{a_n-ib_n}{a_n^{2}+b_n^{2}}-1\right)e^{-(1+a_n+ib_n)/2\eta_n}\right)\rightarrow \pm 1.
\end{equation}
Since $\rp (\bar z^2) \in (1, +\infty]$ by assumption, it follows that there exists a positive constant $c$ such that $a_n^2 > c + 1 + b_n^2$ for $n$ sufficiently large, and in particular that $|\rp (\bar z)| >1$. Let us consider successively the two possibilities $\rp (\bar z) > 1$ and $\rp (\bar z) < -1$.\\

\noi \textbf{The case} $\mathbf{\rp (\bar z) > 1.}$

In this case, if $n$ is sufficiently large, we have
\begin{equation}
\text{Re}\left(\frac{a_n-ib_n}{a_n^{2}+b_n^{2}}+1\right)\geq 1 \hd
\text{and} \hd e^{(a_n-1)/2\eta_n}\longrightarrow +\infty.
\end{equation}
Thus
\begin{equation}
\left|\frac{1}{2}\left(\frac{a_n-ib_n}{a_n^{2}+b_n^{2}}+1\right)e^{(a_n-1)/2\eta_n}\right|\longrightarrow +\infty,
\end{equation}
while
\begin{equation}
\frac{1}{2}\left(\frac{a_n-ib_n}{a_n^{2}+b_n^{2}}-1\right)e^{(-a_n-1)/2\eta_n}e^{-ib_n/2\eta_n}\longrightarrow 0.
\end{equation}
This implies
\begin{equation}
\label{cont-lim-infty}
\left|\frac{1}{2}\left(\frac{a_n-ib_n}{a_n^{2}+b_n^{2}}+1\right)e^{(a_n-1)/2\eta_n}e^{ib_n/2\eta_n}-
\frac{1}{2}\left(\frac{a_n-ib_n}{a_n^{2}+b_n^{2}}-1\right)e^{-(1+a_n+ib_n)/2\eta_n}\right|\longrightarrow +\infty,
\end{equation}
which is in contradiction with \eqref{contralem21}.\\

\noi \textbf{The case} $\mathbf{\rp (\bar z) < - 1.}$

Similarly in this case we have
\begin{equation}
e^{(-a_n-1)/2\eta_n} \longrightarrow +\infty  \hd \text{and} \hd \left|\left(\frac{a_n-ib_n}{a_n^{2}+b_n^{2}}-1\right)e^{(-a_n-1)/2\eta_n}e^{-ib_n/2\eta_n}\right|\longrightarrow +\infty,
\end{equation}
while
\begin{equation}
\label{eq:contlem9}
\left|\frac{1}{2}\left(\frac{a_n-ib_n}{a_n^{2}+b_n^{2}}+1\right)e^{(a_n-1)/2\eta_n}e^{ib_n/2\eta_n}\right| \longrightarrow 0.
\end{equation}
So we get again \eqref{cont-lim-infty} and a contradiction with \eqref{contralem21}. This concludes the proof of Lemma \ref{lemma1}, Part a).\\

\noi \textbf{Part b)}

We assume that
\begin{equation}
\label{rez2=1}
\rp (\bar z^2) = 1.
\end{equation}
We have
\begin{equation}
\rp (\bar z^2) = \big(\rp (\bar z) \big)^{\! 2} - \big(\ip (\bar z) \big)^{\! 2}.
\end{equation}
 From this expression and \eqref{rez2=1},  either $|\rp (\bar z)| = 1$ and $\ip (\bar z) = 0$, or $|\rp (\bar z)| > 1$.

We shall show by contradiction that $|\rp (\bar z)| > 1$ is actually not possible. Let us thus assume that $\rp (\bar z) > 1$ (the case $\rp (\bar z) < - 1$ can be easily handled by symmetry). In that case, using again the notations $a_n = \rp (z_n)$ and $b_n = \ip (z_n)$ and multiplying both sides of \eqref{eq22} by $e^{(1-a_n)/4\eta_n}$, we obtain:
\begin{equation}
\label{expnrezgrand}
\begin{split}
&\left(\frac{a_n-ib_n}{a_n^{2}+b_n^{2}}+1\right)e^{(a_n-1)/4\eta_n + ib_n/2\eta_n}-\left(\frac{a_n-ib_n}{a_n^{2}+b_n^{2}}-1\right)e^{-(1+3a_n+2ib_n)/4\eta_n}\\[0.5em]
&= -e^{(2-a_n^2+b_n^2-a_n-2ia_nb_n)/4\eta_n}\pm\sqrt{e^{(2-a_n^2+b_n^2-a_n-2ia_nb_n)/2\eta_n}+4e^{(1-a_n)/2\eta_n}}.
\end{split}
\end{equation}
Since
\begin{equation}
\label{an>1}
1-a_n \rightarrow 1- \rp (\bar z) < 0
\end{equation}
 and, by \eqref{rez2=1},   $2-a_n^2+b_n^2-a_n \rightarrow 1- \rp (\bar z) < 0$, the right-hand side of \eqref{expnrezgrand} converges to 0 when $\eta_n \rightarrow 0^{+}$. This implies that the left-hand side
of \eqref{expnrezgrand} also converges to 0, namely
\begin{equation}
\left(\frac{a_n-ib_n}{a_n^{2}+b_n^{2}}+1\right)e^{(a_n-1)/4\eta_n + ib_n/2\eta_n}-\left(\frac{a_n-ib_n}{a_n^{2}+b_n^{2}}-1\right)e^{-(1+3a_n+2ib_n)/4\eta_n}\rightarrow 0.
\end{equation}
Then, since the $e^{-(1+3a_n+ib_n)/2\eta_n} \rightarrow 0$, we should have
\begin{equation}
\label{eq:conlem2220}
\left|\left(\frac{a_n-ib_n}{a_n^{2}+b_n^{2}}+1\right)\right|e^{(a_n-1)/4\eta_n}\rightarrow 0,
\end{equation}
but it can be shown that this is impossible. Indeed, by \eqref{an>1}, we know that $a_n > 0$ if $n$ is large enough, which implies that
\begin{equation}
\label{ineqan}
\left|\left(\frac{a_n-ib_n}{a_n^{2}+b_n^{2}}+1\right)\right|e^{(a_n-1)/4\eta_n} \gs e^{(a_n-1)/4\eta_n}.
\end{equation}
From \eqref{an>1}, $e^{\frac{(a_n-1)}{4\eta}}\rightarrow +\infty$
which leads to a contradiction with \eqref{eq:conlem2220} and \eqref{ineqan}.\\

\noi This concludes the proof of Lemma \ref{lemma1}. \qed


\begin{thebibliography}{10}

\bibitem{Aam13}
O-M. Aamo.
\newblock Disturbance rejection in {2 $\times$ 2} linear hyperbolic systems.
\newblock {\em IEEE Transactions on Automatic Control}, 58(5):1095--1106, May
  2013.

\bibitem{AnfAam17}
H.~Anfinsen and O-M. Aamo.
\newblock Adaptive output-feedback stabilization of linear {2 $\times$ 2}
  hyperbolic systems using anti-collocated sensing and control.
\newblock {\em Systems and Control Letters}, 104:86--94, 2017.

\bibitem{AnfAam17b}
H.~Anfinsen and O-M. Aamo.
\newblock Disturbance rejection in general heterodirectional 1-d linear
  hyperbolic systems using collocated sensing and control.
\newblock {\em Automatica}, 76:230--242, 2017.

\bibitem{AurDiMeg}
J.~Auriol and F.~{Di Meglio}.
\newblock Minimum time control of heterodirectional linear coupled hyperbolic
  {PDEs}.
\newblock {\em Automatica}, 71:300--307, 2016.

\bibitem{BasCor14}
G.~Bastin and J-M. Coron.
\newblock {\em Stability and Boundary Stabilisation of 1-D Hyperbolic Systems}.
\newblock Number~88 in Progress in Nonlinear Differential Equations and Their
  Applications. Springer International, 2016.

\bibitem{BasCorHay22a}
G.~Bastin, J-M. Coron, and A.~Hayat.
\newblock The usefulness of diffusion for the robustness of boundary output
  feedback control of an unstable fluid system.
\newblock Work in progress, 2022.

\bibitem{BerKrs14b}
P.~Bernard and M.~Krstic.
\newblock Adaptive output-feedback stabilization of non-local hyperbolic
  {PDE}s.
\newblock {\em Automatica}, 50:2692--2699, 2014.

\bibitem{ChiMazSig16}
Y. Chitour, G. Mazanti, and M. Sigalotti.
\newblock Stability of non-autonomous difference equations with applications to
  transport and wave propagation on networks.
\newblock {\em Netw. Heterog. Media}, 11(4):563--601, 2016.

\bibitem{CorBasdAn08}
J-M. Coron, G.~Bastin, and B.~d'Andr\'ea{-}Novel.
\newblock Dissipative boundary conditions for one dimensional nonlinear
  hyperbolic systems.
\newblock {\em SIAM Journal of Control and Optimization}, 47(3):1460--1498,
  2008.

\bibitem{SanPri08}
V.~{Dos Santos Martins} and C.~Prieur.
\newblock Boundary control of open channels with numerical and experimental
  validations.
\newblock {\em IEEE Transactions on Control Systems Technology},
  16(6):1252--1264, 2008.

\bibitem{FraCol15}
C.~Franco and J.~Collado.
\newblock Ziegler paradox and periodic coefficient differential equations.
\newblock In {\em Proceedings 12th International Conference on Electrical
  Engineering, Computing Science and Automatic Control (CCE)}, pages 1--5,
  2015.

\bibitem{Fri14}
E.~Fridman.
\newblock {\em Introduction to time-delay systems: Analysis and control}.
\newblock Springer, 2014.

\bibitem{Gug15}
M.~Gugat.
\newblock {\em Optimal Boundary Control and Boundary Stabilization of
  Hyperbolic Systems}.
\newblock SpringerBriefs in Electrical and Computer Engineering. Springer,
  2015.

\bibitem{GugLeuTam12}
M.~Gugat, G.~Leugering, S.~Tamasoiu, and K.~Wang.
\newblock ${H}^2$-stabilization of the isothermal euler equations: a {L}yapunov
  function approach.
\newblock {\em Chinese Annals of Mathematics. Series B}, 33(4):479--500, 2012.

\bibitem{HalVer93}
J.K. Hale and S.M. Verduyn{-}Lunel.
\newblock {\em Introduction to Functional-Differential Equations}.
\newblock Number~99 in Appl. Math. Sci. Springer-Verlag, New York, 1993.

\bibitem{HaySha19}
A.~Hayat and P.~Shang.
\newblock A quadratic {L}yapunov function for {S}aint-{V}enant equations with
  arbitrary friction and space-varying slope.
\newblock {\em Automatica}, 100:52--60, 2019.

\bibitem{HuMegVaz15}
L.~Hu, F.~{Di Meglio}, R.~Vazquez, and M.~Krstic.
\newblock Control of homodirectional and general heterodirectional linear
  coupled hyperbolic {PDEs}.
\newblock {\em IEEE Transactions on Automatic Control}, 61(11):3301--3314,
  2016.

\bibitem{KrsGuoBal08}
M.~Krstic, B-Z. Guo, A.~Balogh, and A.~Smyshlyaev.
\newblock Output-feedback stabilization of an unstable wave equation.
\newblock {\em Automatica}, 44:63--74, 2008.

\bibitem{Lic08}
M.~Lichtner.
\newblock Spectral mapping theorem for linear hyperbolic systems.
\newblock {\em Proceedings of the American Mathematical Society},
  136(6):2091--2101, 2008.

\bibitem{MicNic07}
W.~Michiels and S-I. Niculescu.
\newblock {\em Stability and Stabilization of Time-Delay Systems}.
\newblock Advances in Design and Control 12. SIAM, Philadelphia, 2007.

\bibitem{Sil76}
R.A. Silkowski.
\newblock {\em Star Shaped Regions of Stability in Hereditary Systems}.
\newblock PhD thesis, Brown University, Providence, R.I., June 1976.

\bibitem{Tur52}
A.M. Turing.
\newblock The chemical basis of morphogenesis.
\newblock {\em Philosophical Transactions of the Royal Society B},
  237(641):37--72, 1952.

\bibitem{WurMayWoi21}
J.~Wurm, L.~Mayer, and F.~Woittenek.
\newblock Feedback control of water waves in a tube with moving boundary.
\newblock {\em European Journal of Control}, 62:151--157, 2021.

\end{thebibliography}
\end{document}